\documentclass[11pt]{amsart}

\usepackage[all,cmtip,color,matrix,arrow]{xy}
\usepackage{anysize} \marginsize{1.3in}{1.3in}{1in}{1in}
\usepackage{comment}
\usepackage{xcolor}
\usepackage{color}
\usepackage[]{xy}
\usepackage{amsmath,tabu}
\usepackage{mathtools}
\usepackage[utf8]{inputenc}
\usepackage{varioref}
\usepackage{amsfonts}
\usepackage{amssymb}
\usepackage{bbm}
\usepackage{esint}
\usepackage{graphicx}
\usepackage{tikz}
\usepackage{empheq}
\usepackage{enumitem}
\usepackage{tikz-cd}
\usetikzlibrary{matrix,arrows,decorations.pathmorphing}
\usepackage{mathrsfs}
\usepackage[hypertexnames=false,backref=page,pdftex,
 	pdfpagemode=UseNone,
 	breaklinks=true,
 	extension=pdf,
 	colorlinks=true,
 	linkcolor=blue,
 	citecolor=red,
 	urlcolor=blue,
 ]{hyperref}

\newcommand{\sF}{{\mathcal F}}
\newcommand{\sG}{{\mathcal G}}

\newcommand{\sL}{{\mathcal L}}

\newcommand{\sO}{{\mathcal O}}

\newcommand{\sQ}{{\mathcal Q}}

\newcommand{\sU}{{\mathcal U}}
\newcommand{\sV}{{\mathcal V}}

\newcommand{\C}{{\mathbb C}}

\renewcommand{\P}{{\mathbb P}}
\newcommand{\Q}{{\mathbb Q}}
\newcommand{\R}{{\mathbb R}}

\newcommand{\Z}{{\mathbb Z}}

\newcommand{\liesp}{{\mathfrak{sp}}}

\newcommand{\Bl}{{\operatorname{Bl}}}

\newcommand{\Ext}{{\operatorname{Ext}}}

\newcommand{\Hom}{\operatorname{Hom}}

\renewcommand{\Im}{\operatorname{Im}}

\renewcommand{\O}{{\rm O}}

\newcommand{\ol}[1]{{\overline{#1}}}

\newcommand{\rk}{{\rm rk}}

\newcommand{\Sym}{{\rm Sym}}

\renewcommand{\to}[1][]{\xrightarrow{\ #1\ }}

\newcommand{\veps}{\varepsilon}

\newcommand{\wh}[1]{{\widehat{#1}}}
\newcommand{\wt}[1]{{\widetilde{#1}}}

\makeatletter
\newcommand*{\da@rightarrow}{\mathchar"0\hexnumber@\symAMSa 4B }
\newcommand*{\da@leftarrow}{\mathchar"0\hexnumber@\symAMSa 4C }
\newcommand*{\xdashrightarrow}[2][]{%
  \mathrel{%
    \mathpalette{\da@xarrow{#1}{#2}{}\da@rightarrow{\,}{}}{}%
  }%
}
\newcommand{\xdashleftarrow}[2][]{%
  \mathrel{%
    \mathpalette{\da@xarrow{#1}{#2}\da@leftarrow{}{}{\,}}{}%
  }%
}
\newcommand*{\da@xarrow}[7]{%
  \sbox0{$\ifx#7\scriptstyle\scriptscriptstyle\else\scriptstyle\fi#5#1#6\m@th$}%
  \sbox2{$\ifx#7\scriptstyle\scriptscriptstyle\else\scriptstyle\fi#5#2#6\m@th$}%
  \sbox4{$#7\dabar@\m@th$}%
  \dimen@=\wd0 %
  \ifdim\wd2 >\dimen@
    \dimen@=\wd2 %
  \fi
  \count@=2 %
  \def\da@bars{\dabar@\dabar@}%
  \@whiledim\count@\wd4<\dimen@\do{%
    \advance\count@\@ne
    \expandafter\def\expandafter\da@bars\expandafter{%
      \da@bars
      \dabar@ 
    }%
  }%
  \mathrel{#3}%
  \mathrel{%
    \mathop{\da@bars}\limits
    \ifx\\#1\\%
    \else
      _{\copy0}%
    \fi
    \ifx\\#2\\%
    \else
      ^{\copy2}%
    \fi
  }%
  \mathrel{#4}%
}
\makeatother

\makeatletter
\newsavebox\myboxA
\newsavebox\myboxB
\newlength\mylenA

\newcommand*\xtilde[2][0.8]{%
    \sbox{\myboxA}{$\m@th#2$}%
    \setbox\myboxB\null
    \ht\myboxB=\ht\myboxA%
    \dp\myboxB=\dp\myboxA%
    \wd\myboxB=#1\wd\myboxA
    \sbox\myboxB{$\m@th\widetilde{\copy\myboxB}$}
    \setlength\mylenA{\the\wd\myboxA}
    \addtolength\mylenA{-\the\wd\myboxB}%
    \ifdim\wd\myboxB<\wd\myboxA%
       \rlap{\hskip 0.5\mylenA\usebox\myboxB}{\usebox\myboxA}%
    \else
        \hskip -0.5\mylenA\rlap{\usebox\myboxA}{\hskip 0.5\mylenA\usebox\myboxB}%
    \fi}

\newbox\usefulbox

\def\getslant #1{\strip@pt\fontdimen1 #1}

\def\xxtilde #1{\mathchoice
 {{\setbox\usefulbox=\hbox{$\m@th\displaystyle #1$}%
    \dimen@ \getslant\the\textfont\symletters \ht\usefulbox
    \divide\dimen@ \tw@ 
    \kern\dimen@ 
    \xtilde{\kern-\dimen@ \box\usefulbox\kern\dimen@ }\kern-\dimen@ }}
 {{\setbox\usefulbox=\hbox{$\m@th\textstyle #1$}%
    \dimen@ \getslant\the\textfont\symletters \ht\usefulbox
    \divide\dimen@ \tw@ 
    \kern\dimen@ 
    \xtilde{\kern-\dimen@ \box\usefulbox\kern\dimen@ }\kern-\dimen@ }}
 {{\setbox\usefulbox=\hbox{$\m@th\scriptstyle #1$}%
    \dimen@ \getslant\the\scriptfont\symletters \ht\usefulbox
    \divide\dimen@ \tw@ 
    \kern\dimen@ 
    \xtilde{\kern-\dimen@ \box\usefulbox\kern\dimen@ }\kern-\dimen@ }}
 {{\setbox\usefulbox=\hbox{$\m@th\scriptscriptstyle #1$}%
    \dimen@ \getslant\the\scriptscriptfont\symletters \ht\usefulbox
    \divide\dimen@ \tw@ 
    \kern\dimen@ 
    \xtilde{\kern-\dimen@ \box\usefulbox\kern\dimen@ }\kern-\dimen@ }}%
 {}}

\newcommand*\xoverline[2][0.75]{%
    \sbox{\myboxA}{$\m@th#2$}%
    \setbox\myboxB\null
    \ht\myboxB=\ht\myboxA%
    \dp\myboxB=\dp\myboxA%
    \wd\myboxB=#1\wd\myboxA
    \sbox\myboxB{$\m@th\overline{\copy\myboxB}$}
    \setlength\mylenA{\the\wd\myboxA}
    \addtolength\mylenA{-\the\wd\myboxB}%
    \ifdim\wd\myboxB<\wd\myboxA%
       \rlap{\hskip 0.5\mylenA\usebox\myboxB}{\usebox\myboxA}%
    \else
        \hskip -0.5\mylenA\rlap{\usebox\myboxA}{\hskip 0.5\mylenA\usebox\myboxB}%
    \fi}

\def\xxoverline #1{\mathchoice
 {{\setbox\usefulbox=\hbox{$\m@th\displaystyle #1$}%
    \dimen@ \getslant\the\textfont\symletters \ht\usefulbox
    \divide\dimen@ \tw@ 
    \kern\dimen@ 
    \overline{\kern-\dimen@ \box\usefulbox\kern\dimen@ }\kern-\dimen@ }}
 {{\setbox\usefulbox=\hbox{$\m@th\textstyle #1$}%
    \dimen@ \getslant\the\textfont\symletters \ht\usefulbox
    \divide\dimen@ \tw@ 
    \kern\dimen@ 
    \xoverline{\kern-\dimen@ \box\usefulbox\kern\dimen@ }\kern-\dimen@ }}
 {{\setbox\usefulbox=\hbox{$\m@th\scriptstyle #1$}%
    \dimen@ \getslant\the\scriptfont\symletters \ht\usefulbox
    \divide\dimen@ \tw@ 
    \kern\dimen@ 
    \xoverline{\kern-\dimen@ \box\usefulbox\kern\dimen@ }\kern-\dimen@ }}
 {{\setbox\usefulbox=\hbox{$\m@th\scriptscriptstyle #1$}%
    \dimen@ \getslant\the\scriptscriptfont\symletters \ht\usefulbox
    \divide\dimen@ \tw@ 
    \kern\dimen@ 
    \xoverline{\kern-\dimen@ \box\usefulbox\kern\dimen@ }\kern-\dimen@ }}%
 {}}
\makeatother

\makeatletter
\newcommand{\mylabel}[2]{#2\def\@currentlabel{#2}\label{#1}}
\makeatother

\makeatletter
\newcommand{\Mac}{}
\DeclareRobustCommand{\Mac}{%
  M%
  \raisebox{\dimexpr\fontcharht\font`M-\height}{%
    \check@mathfonts\fontsize{\sf@size}{0}\selectfont
    c%
  }%
}
\makeatother 
\newtheoremstyle{citing}
  {}
  {}
  {\itshape}
  {}
  {\bfseries}
  {\textbf{.}}
  {.5em}
  {\thmnote{#3}}

\theoremstyle{plain}

\newtheorem{theorem}[subsection]{Theorem}

\newtheorem{lemma}[subsection]{Lemma}
\newtheorem{corollary}[subsection]{Corollary}

\newtheorem{proposition}[subsection]{Proposition}

\theoremstyle{remark}
\newtheorem{example}[subsection]{Example}

\theoremstyle{definition}

\newtheorem{notation}[subsection]{Notation}

\numberwithin{equation}{section}

\theoremstyle{remark}
\newtheorem{remark}[subsection]{Remark}
\newtheorem*{claim}{Claim}

{\theoremstyle{citing}
}

{\theoremstyle{definition}
}

\title[Hodge structure of O'Grady's Spaces]{Hodge structure of O'Grady's Singular Moduli Spaces}

\author{Valeria Bertini}
\address{Valeria Bertini \\ Fakult\"at f\"ur Mathematik\\ Technische Universit\"at Chemnitz\\
Reichenhainer Stra\ss e 39, 09126 Chemnitz, Germany}
\email{valeria.bertini@mathematik.tu-chemnitz.de}

\author{Franco Giovenzana}
\address{Franco Giovenzana\\ Fakult\"at f\"ur Mathematik\\ Technische Universit\"at Chemnitz\\
Reichenhainer Stra\ss e 39, 09126 Chemnitz, Germany}
\email{franco.giovenzana@mathematik.tu-chemnitz.de}

\let\origmaketitle\maketitle
\def\maketitle{
  \begingroup
  \def\uppercasenonmath##1{} 
  \let\MakeUppercase\relax 
  \origmaketitle
  \endgroup
}

\begin{document}
\thispagestyle{empty}

\begin{abstract}
We investigate the Hodge structure of the singular O'Grady's six and ten dimensional examples of irreducible symplectic varieties. In particular, we compute some of their Betti numbers and their Euler characteristic. As consequence, we deduce that these varieties do not have finite quotient singularities answering a question of Bakker and Lehn.
\end{abstract}

\makeatletter
\@namedef{subjclassname@2020}{
	\textup{2020} Mathematics Subject Classification}
\makeatother

\subjclass[2020]{32J27 (primary), 32S15 (secondary).}
\keywords{Irreducible symplectic varieties, Hodge structure, O'Grady's varieties}

\maketitle

\setlength{\parindent}{1em}
\setcounter{tocdepth}{1}






\tableofcontents

\section{Introduction and notations}\label{section introduction}

\thispagestyle{empty}

Let $X$ be a K3 surface or an abelian surface and $H$ an ample divisor on it. We fix the Mukai vector $v=(2,0,-2)\in H^{2*}(X,\Z)$ and let $M_v(X,H)$ be the associated moduli space of \(S\)-equivalence classes of semi-stable sheaves, where we assume that $H$ is a $v$-generic polarization on \(X\). If \(X=S\) is a K3 surface then $M:=M_v(S,H)$ is O'Grady's singular moduli space of dimension $10$, see \cite{OG99}. If \(X=A\) is an abelian surface we fix \([F_0]\in M_v(A,H)\), where we assume that \([F_0]\) is the \(S\)-equivalence class of the semistable sheaf \(F_0\), and consider the isotrivial fibration 
\begin{align*}
    Alb: M_v(A,H)&\to A\times A^\vee \\
    [F]&\mapsto \bigl(\sum c_2(F), \det(F)\otimes \det(F_0)^{-1}\bigl)
\end{align*}
where the summation is with respect to the group structure of \(A\). The fiber \(K:=Alb^{-1}(0_A,\sO_A)\) is O'Grady's singular moduli space of dimension 6, see \cite{OG03}.

Let \(\wt M\) and \(\wt K\) be the blow-up of \(M\) and \(K\) along their singular loci; Lehn and Sorger \cite[Théorème~1.1]{LS06} proved that \(\wt M\) and \(\wt K\) are smooth irreducible holomorphically symplectic varieties, i.e.  simply connected compact K\"ahler manifolds such that the space of global holomorphic 2-forms is generated by a symplectic form. Furthermore, the manifolds \(\wt M\) and \(\wt K\) are the smooth 10 dimensional and 6 dimensional example by O'Grady \cite{OG99}, \cite{OG03}. 

The Hodge structure of the manifolds \(\wt M\) and \(\wt K\) has been recently computed by De Cataldo, Rapagnetta, Saccà \cite{dCRS21} and Mongardi, Rapagnetta, Saccà \cite{MRS18} respectively. The purpose of this paper is to study the rational cohomology groups of the singular varieties \(K\) and \(M\), and in particular their Hodge structures.

\begin{theorem}\label{thm Betti of M}
Let $\pi\colon \wt M \to M$ be the symplectic resolution of singularities mentioned above.
\begin{itemize}
    \item The pullback $$\pi^* \colon H^{2k}(M, \Q) \to H^{2k}(\wt M, \Q) $$
    is injective for any $k$.
    \item The cohomology groups
    $$H^{2k}(M, \Q) \mathrm{\ and \ }H^{2k + 1}(M, \Q)$$
    carry a pure Hodge structure of weight $2k$.
    \item  We have the following Betti numbers of $M$: 
    \begin{equation*}
\begin{tabu}{ c|c|c|c|c|c|c|c|c|c|c|c } 
& b_0 & b_1 & b_2 & b_3 & b_4 & b_5  & b_{16} & b_{17} & b_{18}  & b_{19} & b_{20} \\
\hline
M & 1 & 0 & 23 & 0 & 276 & 0 & 277 & 0 & 23 & 0 & 1 \\
\hline
\end{tabu}
\end{equation*}
\item The Euler characteristic of $M$ is \(\chi(M)=123606\).
\end{itemize}
\end{theorem}
Even though we do not determine all Betti numbers, we compute some bounds on all of them, cfr. Corollaries \ref{cor Betti estimates 1}, \ref{cor Betti estimates 2}. Regarding the second point of the theorem above, observe that we do not exclude that all the cohomology groups \(H^{2k+1}(M,\Q)\) vanish and hence that all these Hodge structures are trivial. In fact, the odd Betti numbers we have determined so far are all zero.

\begin{theorem}\label{thm Betti of K}
Let \(\pi:\wt K\to K\) be the symplectic resolution of singularities mentioned above.
\begin{itemize}
\item The pullback 
\[
\pi^*: H^{2k}(K, \Q)\to H^{2k}(\wt K, \Q)
\]
is injective for any \(k\).
\item The cohomology groups $$H^{2k}(K, \Q) \mathrm{\ and \ }H^{2k + 1}(K, \Q)$$
    carry a pure Hodge structure of weight $2k$.
\item We have the following Betti numbers of \(K\):
\begin{center}
\begin{tabu}{ c|c|c|c|c|c|c|c  } 
& \(b_0\) & \(b_1\) & \(b_2\) & \(b_3\) & \(b_{10}\) & \(b_{11}\) & \(b_{12}\) \\ 
\hline
\(K\) & 1 & 0 & 7 & 0 & 7 & 0 & 1 \\
\hline
\end{tabu}
\end{center}
\item We have \(b_5(K)\neq 0\).
\item The Euler characteristic of \(K\) is \(\chi(K)=1208\).
\end{itemize}
\end{theorem}

Also in this case, we determine some bounds on the remaining Betti numbers, cfr. Corollary \ref{cor Betti estimates 1 OG6} and Proposition \ref{prop Betti estimates 2 OG6}. Observe that again we do not exclude that most of the odd-cohomology groups are trivial, and indeed at the moment the only non-trivial one is \(H^5(K,\Q)\). The reader should not be surprised of the fact that $H^5(K, \Q)$ carries a Hodge structure of the ``wrong" weight (see Example \ref{example}). 

\begin{remark}\label{rem introduction LLV}
Looijenga, Lunts \cite{LL} and Verbitsky \cite{Verbitsky} studied the cohomology ring of a Hyperk\"ahler manifold $X$ as a representation of a Lie algebra, named after them LLV algebra, isomorphic to $\mathfrak{so}(V,q)$, where 
$$(V,q) = \left( H^2(X,\Q) \oplus \Q^2, q_{BB}\oplus\begin{pmatrix} 0 &1 \\ 1& 0\end{pmatrix}\right)$$ is the extended Mukai lattice. Here $q_{BB}$ is the Beauville-Bogomolov form of \(X\). The decomposition in irreducible representations of $H^*(\wt M, \Q)$ and  $H^*(\wt K, \Q)$ is known \cite[Theorem 1.2]{GKLR} and one direct summand consists of the cohomology algebra generated by $H^2(\wt M, \Q)$ and $H^2(\wt K, \Q)$ respectively, the so-called Verbitsky component \cite[Definition 2.19]{GKLR}.
Our results on the Betti numbers of $M$ and $K$ have been obtained by analysing just the intersection of the Verbitsky component of \(\wt M\) and \(\wt K\) with the image of the pullback \(\pi^*:H^*(M,\Q)\to H^*(\wt M,\Q)\) and \(\pi^*:H^*(K,\Q)\to H^*(\wt K,\Q)\) respectively; in future work we plan to investigate the other components too (cfr. Remark \ref{Remark LLV}).

Moreover, an LLV algebra for the intersection cohomology of varieties admitting a symplectic resolution of singularities has been recently introduced and computed \cite[Theorem~0.5]{FSY22}. It would be interesting to compute their decompositions in irreducible representations of the intersection cohomology of $M$ and $K$ with respect to their LLV algebras.
\end{remark}

\begin{corollary}
The singularities of the varieties $M$ and $K$ are \textbf{not}  finite quotient singularities.
\end{corollary}
\begin{proof}
Theorem \ref{thm Betti of M} shows that Poincar\'e duality fails for \(H^4(M,\Q)\) and \(H^{16}(M,\Q)\), and Theorem \ref{thm Betti of K} shows that the \(k^{th}\)-rational cohomology group of \(K\) does not always carry a pure Hodge structure of weight \(k\), as \(H^5(K,\Q)\) carries a non trivial  Hodge structure of weight 4. This is enough to conclude that \(M\) and \(K\) do not have finite quotient singularieties, see \cite[\S 1.3]{dCM-survey}.
\end{proof}

The varieties \(M\) and \(K\) are known to be \(\Q\)-factorial \cite[Theorem 1.1]{Perego10}. Bakker and Lehn asked about the existence of \(\Q\)-factorial symplectic varieties not having finite quotient singularities. This is an interesting class of varieties in the framework of the recent results on Torelli theorems for singular symplectic varieties: the surjectivity of the period map has been proved by Bakker and Lehn  \cite[Theorem 1.1]{BL20} for primitive symplectic varieties with \(\Q\)-factorial singularities and 2nd Betti number greater or equal than 5, and by Menet \cite[Theorem 1.1]{Menet20} without assumption on the 2nd Betti number but  with the further assumption that the singularities are of finite quotient type.

\begin{remark}\label{remark Mirko}
It is an easy consequence of \cite[Remark~6.3]{KLS06} and \cite[Proposition~5.15]{KollarMori} that the moduli spaces \(M_v(X,H)\) such that \(v=kw\) with \(k\ge 3\) or \(w^2>2\) are examples of \(\Q\)-factorial symplectic varieties not having finite quotient singularities, as they are not analytically \(\Q\)-factorial.

Notice that the case of the varieties \(M\) and \(K\) is different, as they are also analytically \(\Q\)-factorial, see Remark \ref{rem anal Q fact}. Furthermore, the varieties \(M\) and \(K\) give examples of (analytically) \(\Q\)-factorial varieties not having quotient singularities but having a symplectic resolution, differently from the above mentioned \(M_v(X,H)\), \cite[Theorem~6.2]{KLS06}.
\end{remark}

\subsection*{Notations} We will always work over the field of complex numbers \(\mathbb C\). Given a compact complex variety \(X\) we will denote by \(H^k(X)\) its \(k^{th}\)-cohomology group with rational coefficients.

\subsection*{Acknowledgments}
We thank Christian Lehn for his constant support and expert advice through the entire project. We acknowledge fruitful discussions with Luca Giovenzana. We thank Mirko Mauri for pointing out the examples in 
Remark~\ref{remark Mirko}.
Franco Giovenzana was supported by the DFG through the research grant Le 3093/3-2.


\section{The resolution of the singularity of O'Grady}\label{section lehn sorger}
\thispagestyle{empty}
Consider the Mukai vector \(v=(2,0,-2)\) and the moduli space \(M_v(X,H)\) as in the introduction, with \(X\) a K3 or an abelian surface. The singular points of \(M_v(X,H)\) are the \(S\)-equivalence classes of strictly semistable sheaves; if we write $v=2 v_0$, the singular locus of  $M_v(X,H)$ is isomorphic to the symmetric product \(\Sym^2 M_{v_0}(X,H)\), with singular locus isomorphic to \(M_{v_0}(X,H)\). 

Let \(Y\) be one of O'Grady's singular moduli spaces \(M\) and \(K\) as in the introduction.  We call \(\Sigma:=Y^{sing}\) and \(\Omega:=\Sigma^{sing}\). If \(X=S\) is a K3 surface and \(Y=M\) then \(\Omega\cong M_{v_0}(S,H)\) is deformation equivalent to the Hilbert scheme of two points on a K3 surface, because the vector \(v_0\) is primitive in the Mukai lattice \(H^{2*}(S,H)\) of \(S\). If \(X=A\) is an abelian surface, then \(M_{v_0}(A,H)\cong A\times A^\vee\) and in \(Y=K\) we have that \(\Sigma=K\cap \Sym^2 M_{v_0}(A,H)\) is isomorphic to \((A\times A^\vee)/\pm 1\), with singular locus \(\Omega\) consisting of 256 points. For any detail we refer to \cite[\S 1]{PR13} and \cite[\S 1]{MRS18}.

Finally, we call \(\pi:\wt Y:=Bl_\Sigma Y\to Y\) the blow-up giving the symplectic resolution as discussed in the introduction.

\begin{remark}\label{rem anal Q fact}
The varieties \(M\) and \(K\) are analytically \(\Q\)-factorial. Indeed, they have \(A_1\)-singularities along \(\Sigma\smallsetminus\Omega\) (cfr. the local description below) and they have locally factorial singularities along \(\Omega\) \cite[Remark 6.3,(2)]{KLS06}.
\end{remark}

\subsection{Local structure}\label{subsection local structure}
The local structure of the singularities of O'Grady's moduli spaces was described by M. Lehn and Sorger \cite{LS06}. Mongardi, Rapagnetta and Saccà \cite{MRS18} gave a description of the local structure in terms of blow-ups along smooth subvarieties. We briefly recall these descriptions here and refer to the original works for details.

Let $(V,\omega)$ be a complex symplectic vector space of dimension 4 and denote by $\liesp(V)$ the Lie algebra of the symplectic group of $(V,\omega)$. We define 
\begin{equation}\label{eq definition z}
Z:=\{ A\in \liesp(V) \mid A^2 = 0\}.
\end{equation}
These are exactly the endomorphisms of $\liesp(V)$ of rank $\leq 2$ and the singular locus of $Z$ is given by $\Sigma_Z=\{A\in Z\mid \rk(A)\leq 1\}$; the singular locus of \(\Sigma_Z\) is \(\Omega_Z=\{0\}\). Furthermore, \(Z\) has \(A_1\)-singularities along \(\Sigma_Z\).

\begin{theorem}[{\cite[Théorème~4.5]{LS06}}] There are isomorphisms of analytic germs:
\begin{enumerate}
    \item \((K,p)\cong (Z,0)\), for any point \(p\in\Omega\subset K\).
    \item \((M,p)\cong (Z\times\C^4,0)\), for any point \(p\in \Omega \subset M\).
\end{enumerate}
\end{theorem}

We denote by $G$ the Lagrangian Grassmanian of \(V\). We put
\begin{equation}\label{eq local blow up}
\wt Z:=\{(A,U) \in Z \times G \mid U \subset \ker A\}\subseteq Z\times G
\end{equation}
and call \(\pi_Z:\wt Z\to Z\) the restriction of the projection on the first factor. The restriction of the second projection \(\wt Z\to G\) makes \(\wt Z\) the cotangent bundle of \(G\), hence \(\wt Z\) is smooth and \(\pi_Z:\wt Z\to Z\) is a symplectic resolution of \(Z\). Lehn and Sorger show the following.

\begin{theorem}[{\cite[Théorème~2.1]{LS06}}]
The morphism $\pi_Z:\wt Z \to Z$ defined above  coincides with the blow up of $Z$ in $\Sigma_Z$ equipped with the reduced structure.\qed
\end{theorem}

We call \(\wt\Sigma_Z\) the exceptional divisor of \(\pi_Z\), i.e. 
\[
\wt \Sigma_Z=\{(A,U) \in \Sigma_Z \times G \mid U \subset \ker A\}.
\]
The restriction morphism $\pi_Z|_{\wt\Sigma_Z}:\wt\Sigma_Z \to \Sigma_Z$ is a $\P^1$-fiber bundle outside of \(\Omega_Z\), whereas the fiber $\wt\Omega_Z:=\pi_Z^{-1}(0) = G$ is the Lagrangian Grassmanian.

As observed in \cite[Remark~2.1]{MRS18}, the variety \(\wt Z\) is isomorphic to the total space \(Sym^2_G\mathcal U\), where \(\mathcal U\) is the rank 2 tautological bundle over \(G\); the isomorphism is realized as follows. Given \((A,U)\in \wt\Sigma_Z\), the endomorphism \(A\) factorizes as \(V\to V/U\xrightarrow{\phi_A} U\hookrightarrow V\), and the symplectic form \(\omega\) induces an isomorphism \(V/U\cong U^\vee\) because \(U\) is a Lagrangian subspace. Hence \(\phi_A\in \Hom(U^\vee,U)\cong U\otimes U\cong (U^\vee\otimes U^\vee)^\vee\), where the associated bilinear form \((f,g)\mapsto f(\phi_A(g))\) on \(U^\vee\otimes U^\vee\) is symmetric because \(A\in \mathfrak{sp}(V)\) satisfies \(\omega(Av,w)=\omega(Aw,v)\) for any \(v,w\in V\); hence \(\phi_A\in \Sym^2U\). Under this identification the variety \(\wt\Sigma\) corresponds to the locus parametrizing singular symmetric bilinear forms on the fibers of \(\mathcal U^\vee\), which is a fibration in cones over a smooth conic over \(G\), having singularities only along the zero section \(\wt\Omega_Z\).

Following \cite[Proposition 2.4]{MRS18} we consider the varieties:
\begin{gather*}
\ol Z:=Bl_{\Omega_Z} Z=\{([B],A)\mid A\in [B]\}\subseteq \P(Z)\times Z\\
\wh Z:=Bl_{\wt\Omega_Z}\wt Z=\{([B],A,U)\mid A(U)=0,\ A\in[B]\}\subseteq \P(Z)\times Z\times G
\end{gather*}
with blow-up morphisms \(\rho_Z:\ol Z\to  Z\) and \(\gamma_Z:\wh Z\to \wt Z\). If \(\ol \Sigma_Z\) is the strict transform of \(\Sigma_Z\) in \(\ol Z\) via \(\rho_Z\), i.e. \(\ol \Sigma_Z\cong Bl_{\Omega_Z}\Sigma_Z\), one has \(\wh Z\cong Bl_{\ol \Sigma_Z}\ol Z\) and the following diagram is commutative:
\begin{center}
    \begin{tikzcd}
& \wh Z \ar[dr, "\phi_Z"] \ar[dl, "\gamma_Z"'] & \\
\wt Z \ar[dr, "\pi_Z"'] & & \ol Z \ar[dl, "\rho_Z"] \\
& Z &
\end{tikzcd}
\end{center}
Here \(\phi_Z:\wh Z\to \ol Z\) is the blow-up morphism. Furthermore, \(\ol \Sigma_Z\) is smooth, \(\ol Z\) has \(A_1\) singularity along \(\ol \Sigma_Z\) and \(\wh Z\) is smooth. Observe that the morphisms \(\phi_Z\), \(\gamma_Z\) and \(\rho_Z\) are blow-ups along smooth subvarieties, and the diagram above relates the resolution \(\pi_Z:\wt Z\to Z\) with them. This is exactly in the spirit of the original works \cite{OG99} and \cite{OG03} by O'Grady.

We give names to the transformations of \(\Omega_Z\) and \(\Sigma_Z\) inside the new varieties \(\ol Z\) and \(\wh Z\): We call \(\ol\Omega_{OG,Z}\) the exceptional divisor of \(\ol Z\xrightarrow{\rho_Z}Z\) and we set \(\ol \Omega_Z:=\ol\Omega_{OG,Z}\cap\ol \Sigma_Z\); note that \(\ol\Omega_Z=\{([B],0)\mid [B]\in \P(\Sigma_Z)\}\cong \P(V)\). We denote by \(\wh\Sigma_Z\) the exceptional divisor of \(\wh Z\xrightarrow{\phi_Z}\ol Z\), by \(\wh \Omega_{OG,Z}\) the exceptional divisor of \(\wh Z\xrightarrow{\gamma_Z}\wt Z\) and we set \(\wh \Omega_Z:=\wh \Omega_{OG,Z}\cap \wh\Sigma_Z\subseteq \wh\Omega_{OG,Z}\).

\begin{proposition} \cite[Corollary 2.5]{MRS18}\label{cor 2.5 MRS}
\begin{enumerate}
    \item \(\wh \Sigma_Z\) is a \(\P^1\)-bundle over \(\ol \Sigma_Z\) and \(\wh \Sigma_Z\cong Bl_{\wt\Omega_Z}\wt\Sigma_Z\).
    \item \(\wh \Omega_{OG,Z}\) is a \(\P^2\)-bundle over \(\wt\Omega_Z\) isomorphic to \(\P(Sym^2_G\mathcal U)\).
\end{enumerate}
\end{proposition}
The second part of the proposition above is obtained via the isomorphism \(\wt Z\cong Sym^2\mathcal U\) described before, which realizes \(\wt \Omega_Z\) as the zero section of \(Sym^2\mathcal U\); as consequence \(\hat Z=Bl_{\wt \Omega_Z}\wt Z\) is isomorphic to the total space \(T\subseteq \P(Sym^2_G\mathcal U)\times Sym^2_G\mathcal U\) and it has exceptional divisor \(\wh \Omega_{OG,Z}\cong \P(Sym^2_G\mathcal U)\). 

\begin{corollary}\label{cor wh Omega local}
\(\wh \Omega_Z\) is a \(\P^1\)-bundle over \(\wt\Omega_Z\) isomorphic to \(\P_G(\mathcal U)\), and the natural inclusion \(\wh\Omega_Z\subseteq \wh\Omega_{OG,Z}\) corresponds to the inclusion \(\P_G(\mathcal U)\subseteq \P(Sym^2_G\mathcal U)\) induced by \(U\rightarrow \Sym^2 U\), \(u\mapsto u\otimes u\).
\end{corollary}
\begin{proof}
From Proposition \ref{cor 2.5 MRS} we have that \(\wh \Omega_Z\) is the exceptional divisor of \(\wh\Sigma_Z\cong Bl_{\wt\Omega_Z}\wt\Sigma_Z\). An element in \(\Sigma_Z\) is of the form \(A_v=\omega(-,v)v\) for some \(v\in U\); the isomorphism \(\wt Z\cong Sym^2_G\mathcal U\) sends the element \((A_v, U)\in \wt\Sigma_Z\) to the symmetric bilinear form \(\{(f,g)\mapsto f(v)g(v)\}\) on \(U^\vee\), which lies in the image of \(U\xrightarrow{\alpha} U\otimes U\cong (U^\vee\otimes U^\vee)^\vee\), where \(\alpha\) is the map \(u\mapsto u\otimes u\). Notice that two such symmetric biliner form \(\{(f,g)\mapsto f(v)g(v)\}\) and  \(\{(f,g)\mapsto f(w)g(w)\}\) are equal exactly when \(w=\pm v\), which is the condition to have \(A_v=A_w\). We conclude \(\wt\Sigma_Z\cong\alpha_G(\mathcal U)\cong \mathcal U/\pm 1\subseteq Sym^2_G\mathcal U\cong \wt Z\), where \(\alpha_G\) is the relative version of the \(\alpha\) just described on the fibers and \(\pm 1\) is the fiberwise action; again, \(\wt\Omega_Z\) corresponds to the zero section of \(\mathcal U\). Hence  \(\wh \Sigma_Z\cong Bl_{\wt\Omega_Z}\wt\Sigma_Z\) is isomorphic to the total space \(T\subseteq \P_G(\mathcal U)\times \mathcal U/\pm 1\) and it has exceptional divisor \(\wh \Omega_Z\cong\P_G(\mathcal U)\), with inclusion \(\wh\Omega_Z=\wh\Omega_{OG,Z}\cap\wh\Sigma_Z\subseteq\wh\Omega_{OG,Z}\) as in the statement.
\end{proof}

\begin{proposition}\label{prop whOmega over wtOmega}
\(\wh\Omega_Z\) is a \(\P^1\)-bundle over \(\ol\Omega_Z\) isomorphic to \(\P_{\ol\Omega_Z}(\sL^\perp/\sL)\), where \(\sL\) is the tautological line bundle of \(\ol\Omega_Z\cong \P(V)\).
\end{proposition}
\begin{proof}
We define an isomorphism \(\P_G(\sU)\xrightarrow{\sim}\P_{\P(V)}(\sL^\perp/\sL)\), and the claim will follow from Corollary \ref{cor wh Omega local}. Given a Lagrangian subspace $U$ of $V$ and a line $L$ in $U$, then $U$ is contained in $L^\perp$ and defines a line $\ol U$ in the quotient $L/L^\perp$; the association $(L,U) \mapsto (L, \ol U)$ defines the desired morphism. On the other hand, given a line $L$ in $V$ and a line $W$ in $L^\perp/L$, then the preimage $U$ of $W$ in $L$ along the quotient map $L^\perp/L$ is a Lagrangian subspace of $V$ which contains $L$. Thus, the association $(L,W) \mapsto (L, U)$ defines the inverse morphism.
\end{proof}


\subsection{Global structure}\label{subsection global structure} Here we state the global version of the results in the previous section.  In analogy with the constructions and notations used in the local setting, we consider:
\begin{itemize}
\item \( \overline Y:=Bl_{\Omega}Y\xrightarrow{\rho} Y \)
with exceptional divisor \(\ol\Omega_{OG}\); we call \(\ol\Sigma\) the strict transform of \(\Sigma\) and \(\ol\Omega:=\ol\Omega_{OG}\cap\ol\Sigma\).
\item \(\wh Y:=Bl_{\ol \Sigma}\ol Y\xrightarrow{\phi}\ol Y\) with exceptional divisor \(\wh\Sigma\).
\item \(Bl_{\wt\Omega }\wt Y\xrightarrow{\gamma}\wt Y\) with exceptional divisor \(\wh\Omega_{OG}\); we call \(\wh\Omega:=\wh\Omega_{OG}\cap\wh\Sigma\).
\end{itemize}

The following result is \cite[Proposition 2.6]{MRS18} when \(Y=K\), and exactly the same proof gives the result when \(Y=M\).

\begin{proposition}
The varieties \(\wh Y\) and \(Bl_{\wt\Omega}\wt Y\) are smooth and isomorphic over \(Y\), hence the following diagram is commutative: 
\begin{center}
    \begin{tikzcd}
& \wh Y \ar[dr, "\phi"] \ar[dl, "\gamma"'] & \\
\wt Y \ar[dr, "\pi"'] & & \ol Y \ar[dl, "\rho"] \\
& Y &
\end{tikzcd}
\end{center}
\end{proposition}

We pass to the global structure of the \(\Sigma\)-varieties.
\begin{proposition}\label{prop global Sigma's} Let \(Y\) be the variety \(M\) or \(K\).
\begin{enumerate}
    \item \(q:=\rho|_{\ol\Sigma}:\ol\Sigma\to \Sigma\) is the blow-up of
    \(\Sigma\) in \(\Omega\), with exceptional divisor \(\ol\Omega\).
    \item \(f:=\phi|_{\wh\Sigma}:\wh\Sigma\to\ol\Sigma\) is a \(\P^1\)-fiber bundle.
    \item \(g:=\gamma|_{\wh\Sigma}:\wh\Sigma\to\wt\Sigma\) is the blow-up of \(\wt\Sigma\) in \(\wt\Omega\), with exceptional divisor \(\wh\Omega\).
\end{enumerate} 
In particular, the varieties \(\ol\Sigma\) and \(\wh\Sigma\) are smooth. 
\end{proposition}
\begin{proof}
The case \(Y=K\) \cite[Remark 2.7, Corollary 2.8(1)]{MRS18}. We pass to the case \(Y=M\).
\begin{enumerate}
    \item The statement follows from the fact that \(\Sigma\) contains \(\Omega\) as closed subscheme. \(\ol\Sigma\) is smooth because \(\Sigma\smallsetminus \Omega\) is smooth and \(M\) has \(A_1\) singularities along it, as it follows from the local description of the varieties.
    \item The statement follows from the local one in Proposition \ref{cor 2.5 MRS}(1).
    \item The statement follows again from Proposition \ref{cor 2.5 MRS}(1), as the blow-up is a local construction.
\end{enumerate}
\end{proof}
We conclude with the global structure of the \(\Omega\)-varieties.
\begin{proposition}\label{prop global Omega's 10} Consider the case \(Y=M\).
\begin{enumerate}
    \item \(q_\Omega:=q|_{\ol\Omega}:\ol\Omega\to\Omega\) is a \(\P^3\)-bundle. More precisely, \(\ol\Omega\cong\P_{\Omega}(T\Omega)\) over $\Omega$.
    \item \(f_\Omega:=f|_{\wh\Omega}:\wh\Omega\to\ol\Omega\) is a \(\P^1\)-bundle. More precisely, let $\sL\subset q_\Omega^*T\Omega$ be the tautological subbundle, then  \(\P_{\Omega}(\sL^\perp/\sL) \simeq \wh\Omega\) over $\ol\Omega$.
    \item \(p_\Omega:=\pi|_{\wt\Omega}:\wt\Omega\to\Omega\) is a \(G\)-bundle, where \(G\) is the Lagrangian Grassmannian of 2 dimensional Lagrangian spaces in a symplectic vector space. More precisely, \(\wt\Omega\cong\sL\mathcal G_\Omega(T\Omega) \) over $\Omega$, where \(\sL\mathcal G_\Omega(T\Omega)\) is the relative Lagrangian Grassmannian on \(T\Omega\). 
    \item \(g_\Omega:=g|_{\wh\Omega}:\wh\Omega\to\wt\Omega\) is a \(\P^1\)-bundle. More precisely, \(\wh\Omega\cong \P(\mathcal U)\) over $\wt\Omega$, where \(\mathcal U\) is the tautological bundle of \(\sL\mathcal G_\Omega(T\Omega)\).
\end{enumerate}
In particular, the four varieties \(\Omega\), \(\ol\Omega\), \(\wt\Omega\) and \(\wh\Omega\) are smooth.
\end{proposition}
\begin{proof}
\begin{enumerate}
    \item Since \(\Sigma\cong \Sym^2\Omega\), the blow-up (cfr. Proposition \ref{prop global Sigma's}.(1)) \(\ol\Sigma\xrightarrow{q}\Sigma\) can be identified with the Hilbert-Chow morphism \(\Omega^{[2]}\to \Sym^2\Omega\). The exceptional divisor \(\ol\Omega\) is then isomorphic to \(\P(T\Omega)\).
    \item This follows essentially from the construction of O'Grady and the careful analysis of the various blow-ups. Since the details are technical, we discuss this point in the appendix for sake of the exposition. 
    \item This follows by the original construction of $\wt M$ in O'Grady's work, see \cite[\S 2.2]{OG99}.
    \item By \cite[Proposition 2.0.1]{OG99} the variety \(\wh\Omega_{OG}\) is isomorphic to \(\P_{\wt\Omega}(Sym^2\mathcal U)\) over \(\wt\Omega\cong \sL\mathcal G_\Omega(T\Omega)\), hence from the local description in Corollary \ref{cor wh Omega local} we deduce that $\wh\Omega$ is identified with the rank 1 tensors in \(\P_{\wt\Omega}(Sym^2\mathcal U)\), thus isomorphic to \(\P_{\wt\Omega}(\mathcal U)\).
\end{enumerate}
As \(\Omega\) is smooth, all the \(\Omega\)-varieties are smooth as well.
\end{proof}

Now, let $A$ be an abelian surface. The dimension 10 variety $M_{v}(A,H)$ has an up to scalar unique symplectic form $\sigma$ on its regular part, which by \cite[Theorem~2.3]{Kal06b} induces a symplectic form $\alpha$ on $M_{v_0}(A,H)$. Let $V$ be the restriction $TM_{v_0}(A,H)|_{\Omega}$ of its tangent bundle, which we consider as a symplectic vector bundle equipped with the restriction of $\alpha$.

\begin{proposition}\label{prop global Omega's 6}
Consider the case \(Y=K\).
\begin{enumerate}
    \item \(q_\Omega:=q|_{\ol\Omega}:\ol\Omega\to \Omega\) is the projection of the projective bundle \(\P_\Omega(V)\) on \(\Omega=\) 256 points.
    \item \(f_\Omega:=f|_{\wh\Omega}:\wh\Omega\to\ol\Omega\) is a \(\P^1\)-bundle. More precisely, given $\sL\subset q_\Omega^*V$ the tautological subbundle then  \(\P_{\ol\Omega}(\sL^\perp/\sL) \simeq \wh\Omega\) over $\ol\Omega$.
    \item \(p_\Omega:=\pi|_{\wt\Omega}:\wt\Omega\to\Omega\) is the projection of the Lagrangian bundle \(LG_\Omega(V)\) on \(\Omega=\) 256 points.
    \item \(g_\Omega:=g|_{\wh\Omega}:\wh\Omega\to\wt\Omega\) is a \(\P^1\)-bundle. More precisely, \(\wh\Omega\cong \P(\mathcal U)\) over $\wt\Omega$, where \(\mathcal U\) is the universal bundle of \(G\).
\end{enumerate}
In particular, the four varieties \(\Omega\), \(\ol\Omega\), \(\wt\Omega\) and \(\wh\Omega\) are smooth.
\end{proposition}
\begin{proof}
The proof of Proposition \ref{prop global Sigma's} holds for the variety $M_v(A,H)$ too (cf, \cite[\S 2.1]{OG03}) and the statement follows from the local case Proposition \ref{cor 2.5 MRS}, Corollary \ref{cor wh Omega local} and Proposition \ref{prop whOmega over wtOmega}, as \(\Omega\) consists of 256 points.
\end{proof}

 We fix the notations we are going to use in what follows. 
\begin{notation}\label{notation} 
 Let \(Y\) be one of the varieties \(M\) or \(K\) introduced in Section \ref{section introduction}, consider \(\Sigma=Y^{sing}\), \(\Omega=\Sigma_Y^{sing}\) and the symplectic resolution \(\pi:\wt Y=Bl_\Sigma Y\to Y\), with exceptional divisor \(\wt\Sigma\); \(\wt\Omega\) is the strict transform of \(\Omega\). At the beginning of Section \ref{subsection global structure} we have defined the varieties \(\ol\Sigma,\wh\Sigma\) and the varieties \(\ol\Omega_{OG},\ol\Omega,\wh\Omega_{OG},\wh\Omega\). In short, we will refer to \(\Sigma,\wt\Sigma,\ol\Sigma,\wh\Sigma\) as ``the \(\Sigma\)-varieties" and to  \(\Omega,\wt\Omega,\ol\Omega,\wh\Omega\) as ``\(\Omega\)-the varieties ".
 
 We  have morphisms \(p:=\pi|_{\wt\Sigma}\) and \(f,g,q\) defined in Proposition \ref{prop global Sigma's}, and   \(f_\Omega, g_\Omega,p_\Omega,q_\Omega\) defined in Proposition \ref{prop global Omega's 6}, \ref{prop global Omega's 10}. We call \(i:\Omega\hookrightarrow \Sigma\), \(j:\Sigma\hookrightarrow  Y\) and \(i_Y=j\circ i:\Omega\hookrightarrow Y\) the natural inclusions, and analogously for all the \(\Sigma\)- and \(\Omega\)-varieties, using the corrensponding decoration.
 
 Summarizing, we have the following diagram: 
 \begin{center}
 \begin{tikzcd}[row sep=small]
& & & \wh Y \arrow[dddlll, "\gamma"] \arrow[dddrrr, "\phi"] & & &\\
& & & \wh\Sigma\arrow[u, phantom, sloped, "\subseteq"] \arrow[ddll, "g"] \arrow[ddrr, "f"] & & & \\
& & & \wh\Omega\arrow[u, phantom, sloped, "\subseteq"] \arrow[dl, "g_\Omega"] \arrow[dr, "f_\Omega"] & & & \\
\wt Y \arrow[rrrddd, "\pi"] & \wt\Sigma \arrow[l, phantom, sloped, "\supseteq"] \arrow[rrdd, "p"] & \wt\Omega \arrow[l, phantom, sloped, "\supseteq"] \arrow[dr, "p_\Omega"]  & & \ol\Omega\arrow[r, phantom, sloped, "\subseteq"] \arrow[dl, "q_\Omega"] & \ol\Sigma \arrow[r, phantom, sloped, "\subseteq"] \arrow[ddll, "q"] & \ol Y \arrow[dddlll, "\rho"]\\
& & & \Omega \arrow[d, phantom, sloped, "\subseteq"]& & & \\
& & & \Sigma \arrow[d, phantom, sloped, "\subseteq"] & & & \\
& & & Y & & &
\end{tikzcd}
\end{center}
 \end{notation}

\section{Semisimplicial resolutions}\label{section semisimplicial}

The mixed Hodge structure of a complete algebraic variety can be computed using semisimplicial resolutions \cite{Del74}. Following the algorithm in the proof of \cite[Theorem~5.2.6]{PS08}, we obtain a semisimplicial resolution from a cubical hyperresolution as we will now explain. 

Assume the notations are as in Notation \ref{notation} and consider the diagram
\begin{equation}\label{eq tilde diagram}
\xymatrix{
\wt Y \ar[d]_{\pi} & \wt\Sigma \ar@{_(->}[l]_{\wt j} \ar[d]_{p} & \wt \Omega \ar@{_(->}[l]_{\wt i} \ar[d]_{p_\Omega} \\
Y & \Sigma \ar@{_(->}[l]_{j} & \Omega \ar@{_(->}[l]_{i} \\
}
\end{equation}
As observed in Section \ref{section lehn sorger} the morphism $p:\wt\Sigma\to \Sigma$ is generically a smooth $\P^1$-bundle over the complement of $\Omega$ and $p_\Omega:\wt \Omega \to \Omega$ is a $LG_2$-fiber bundle. To obtain a semisimplicial resolution, we consider the resolution  \(q:\ol\Sigma=\Bl_{\Omega} \Sigma\to\Sigma\) (cfr. Proposition \ref{prop global Sigma's}) and the induced diagram
\begin{equation}\label{eq bar diagram}
\xymatrix{
&\wh\Sigma \ar[dl]_g \ar[dd]_(.3){f}|(0.5)\hole && \wh \Omega \ar@{_(->}[ll]_{\wh i} \ar[dd]_{f_\Omega} \ar[ld]_{g_\Omega} \\
\wt\Sigma \ar[dd]_p && \wt \Omega \ar@{_(->}[ll]_(.3){\wt i} \ar[dd]_(.3){p_\Omega} \\
&\ol\Sigma \ar[ld]_(0.4){q} && \ol \Omega. \ar|(0.5)\hole@{_(->}[ll]_(.3){\ol i} \ar[dl]^{q_\Omega} \\
\Sigma && \Omega \ar@{_(->}[ll]_i \\
}
\end{equation}
Let $Y_{\{\bullet\}}$ be the $2$-cubical variety
\begin{equation}\label{eq cubical hyperresolution}
\xymatrix{
&\wh\Sigma \ar[dl]_{\wt j\circ g} \ar[dd]_(0.3){f}|(0.5)\hole && \wh \Omega \ar@{_(->}[ll]_{\hat i} \ar[dd]_{f_\Omega} \ar[ld]_{g_\Omega} \\
\wt Y\ar[dd]_\pi && \wt \Omega \ar@{_(->}[ll]_(0.3){\wt i} \ar[dd]_(0.3){p_\Omega} \\
&\ol\Sigma \ar[ld]_q && \ol \Omega \ar|(0.5)\hole@{_(->}[ll]_(0.3){\ol i} \ar[dl]^{q_\Omega} \\
Y && \Omega \ar@{_(->}[ll]_i \\
}
\end{equation}
and let $\veps_\bullet:Y_\bullet \to Y$ be the induced augmented semisimplicial variety
\begin{equation}\label{eq semisimplicial hyperresolution}
\xymatrix{
Y_0 \ar[d]_{\veps_0} & Y_1 \ar[ld]|{\veps_1}\ar@<0.3ex>[l]\ar@<-0.3ex>[l]& Y_2 \ar@<0.5ex>[l]\ar[l]\ar@<-0.5ex>[l]\ar[lld]^{\veps_2}\\
Y&&\\
}
\end{equation}
where
\[
Y_0=\wt Y \,\amalg\, \ol\Sigma \,\amalg\,  \Omega, \quad Y_1=\wh\Sigma \,\amalg\,  \wt \Omega \,\amalg\,  \ol\Omega, \quad Y_2=\wh\Omega.
\]

\begin{proposition}\label{proposition semisimplicial resolution}
The augmented semisimplicial variety is a semisimplicial hyperresolution of $Y$.
\end{proposition}
\begin{proof}
The very construction of \(\varepsilon_\bullet:Y_\bullet\to Y\) here above is done following the algorithm in the proof of \cite[Theorem~5.26 and its proof]{PS08}, to produce a semisimplicial hyperresolution of \(Y\); the smoothness of the $Y_i$ follows from Proposition~\ref{prop global Sigma's}, Proposition~\ref{prop global Omega's 6} and Proposition~\ref{prop global Omega's 10}.
\end{proof}

From the semisimplicial hyperresolution we get the spectral sequence \cite[Proposition~3.3]{GNAPGP}
\begin{equation}\label{eq. spectral sequence}
    E_1^{p,q} = H^q(Y_p,\Q) \Rightarrow H^{p + q}(Y, \Q) 
\end{equation}
which degenerates at the second page and converges to the cohomology of $Y$ filtered by the weights.
The differentials of the first page are given by an alternating sum of pullbacks along the maps appearing in the resolution.
Spelling out the details in our case, we get that the differential $d_0:E_1^{0,q}\to E_1^{1,q}$ is given by the sum of pullbacks
\begin{equation}\label{eq first differentials}
\xymatrix{
H^q(\widetilde Y)\ar@[blue][r]^{\textcolor{blue}{(\wt j\circ g)^*}}\ar@[blue][rd]^(0.4){\textcolor{blue}{\wt i^*}} & H^q(\wh \Sigma)\\
H^q(\overline\Sigma)\ar@[red]^(0.3){\textcolor{red}{f^*}}[ru]\ar@[blue][rd]^(0.4){\textcolor{blue}{\ol i^*}} & H^q(\widetilde\Omega)\\
H^q( \Omega )\ar@[red][ru]^(0.3){\textcolor{red}{p^*_\Omega}}\ar@[red][r]^{\textcolor{red}{q^*_\Omega}} & H^q(\overline\Omega).
}
\end{equation}
where the red ones are taken with a minus sign. 
Thus, one sees at once that
\begin{lemma}\label{lem E^(0,q)_2}
\begin{align*}
    E^{0,q}_2 = \{(a,b,c) \in H^q(\widetilde Y) \oplus H^q(\overline\Sigma)\oplus H^q(\Omega) | \  (\wt j\circ g)^*a - f^*b &= 0 \in  H^q(\widehat \Sigma),\\ 
 \wt i^* a- p^*_\Omega c &= 0\in H^q(\wt\Omega),\\ 
 \ol i^* b-q^*_\Omega c &= 0 \in H^q(\ol\Omega)\}.
\end{align*}

\end{lemma}

On the other hand, the differential $d_1:E_1^{1,q}\to E^{2,q}_1$ is the difference of blue and red
\begin{equation}\label{eq second differentials}
\xymatrix{
H^q(\widehat \Sigma)\ar@[blue][rd]^{\textcolor{blue}{\wh i^*}}  \\
H^q(\widetilde\Omega)\ar@[red][r]^{\textcolor{red}{g^*_\Omega}} & H^q(\widehat \Omega).\\
H^q(\overline\Omega)\ar@[red][ru]^{\textcolor{red}{f^*_\Omega}}
}
\end{equation}
hence one has the following
\begin{lemma}\label{lem E_2^{1-2,q}}
$E^{2,q}_2$ is the cokernel of 
\begin{align*}
   d_1\colon H^q(\widehat \Sigma) \oplus H^q(\overline\Omega) \oplus H^q(\overline\Omega) &\longrightarrow H^q(\widehat \Omega)\\
    (a,b,c) &\longmapsto \wh i^* a - g^*_\Omega b -f^*_\Omega c.
\end{align*}
\end{lemma}
The following sections are devoted to the computation of the page \(E_2^{p,q}\) when \(Y=M\) and \(K\).
\section{The cohomology of the \(\Sigma\)-varieties and \(\Omega\)-varieties}\label{section cohomology of Sigma and Omega}

Assume the notations as in Notation \ref{notation}. We start the computation of the spectral sequence \eqref{eq. spectral sequence} computing the objects in the page \(E_1^{p,q}\), i.e. the cohomology of the \(\Omega\)- and \(\Sigma\)-varieties. 


\subsection{The cohomology of the $\Omega$-varieties}
We start describing the cohomology of the \(\Omega\)-varieties in the local case presented in  Section \ref{subsection local structure}: Let $(V,\omega)$ be a symplectic vector space of dimension $4$, \(G=LG(V)\) be the Lagrangian Grassmannina of it with tautological bundle \(\sU\) and \(\sL\subseteq V\otimes\sO_{\P(V)}\) be the tautological line subbundle. We have described the following local structure of the \(\Omega\)-varieties:
 \begin{center}
 \begin{tikzcd}
& \wh\Omega_Z \arrow[dl, "\gamma_\Omega"] \arrow[dr, "\phi_\Omega"] & & & & \P_G(\sU)\cong \P_{\P(V)}(\sL^\perp/\sL)\arrow[dl, "\gamma_\Omega"] \arrow[dr, "\phi_\Omega"] & \\
\wt\Omega_Z\arrow[dr, "\pi_\Omega"] & & \ol\Omega_Z\arrow[dl, "\rho_\Omega"] & \cong & G \arrow[dr, "\pi_\Omega"] & & \P(V)\arrow[dl, "\rho_\Omega"] \\
& \Omega_Z & & & & \{0\} &
\end{tikzcd}
\end{center}
where the maps are the restriction of the respective ones to the \(\Omega_Z\)-varieties. Consider the tautological short exact sequence on \(G\) $$0 \to \sU \to V\otimes \sO_G \to \sQ \to 0.$$
Since $U\in G$ is Lagrangian we have \(V/U\cong U^\vee\), hence \(\sQ\cong \sU^\vee\) and
\begin{align*}
    c_1(\sU)^2 -2c_2(\sU) = c_2(\sU)^2 = 0\in H^*(G).
\end{align*}
As $G$ is a 3-dimensional quadric we conclude that
$$H^*(G) \simeq  \frac{\Q[ u_1, u_2] }{(u_1^2 - 2u_2, u_2^2)},$$
where the isomorphism identifies $u_1$ and $u_2$ with the first and the second Chern classes of $\sU$, which have degree $1$ and $2$.  \\
The pullback along $\gamma_\Omega$ defines an injective morphism of rings $H^*(G) \to H^*(\P(\sU))$ and we have
$$H^*(\P_G(\sU)) \simeq \frac{  H^*(G)[ \xi ]}{(\xi^2 - c_1(\sU)\xi + c_2(\sU))} \simeq
\frac{\Q[u_1, u_2, \xi]}{(u_1^2 - 2u_2, u_2^2, \xi^2 - u_1\xi + u_2)},
$$
where the first isomorphism takes the first Chern class of the tautological subbundle of $\gamma_\Omega^*\sU$ to $\xi$. \\
The cohomology ring of the projective space $\P(V)$ is generated by the first Chern class of \(\sL\)  and we have 
$$H^*(\P(V)) \simeq  \frac{\Q[ \zeta]}{(\zeta^4)},\  c_1(\sL) \mapsto \zeta.$$
Next, we look at $\phi_\Omega\colon \P_{\P(V)}(\sL^\perp/ \sL)\to \P(V)$ and we denote by $h$ the first Chern class of the tautological line subbundle of $\phi_\Omega^*(\sL^\perp/ \sL)$; the projective bundle formula gives
$$H^*(\P_{\P(V)}(\sL^\perp /\sL)) \simeq \frac{\phi_\Omega^* H^*(\P(V))[ h ]}{(h^2 - hc_1(\sL^\perp/L) + c_2(\sL^\perp/\sL))}.
$$
Using the two short exact sequences
\begin{gather*}
    0 \to \sL^\perp \to V^\vee\otimes\sO_{\P(V)} \to \sL^\vee \to 0 \\
    0 \to \sL \to \sL^\perp \to \sL^\perp/\sL \to 0
\end{gather*}
one computes
\begin{equation}\label{eq c(L^perp/L)}
    c(\sL^\perp/\sL) = c(\sL^\perp)/c(\sL) = c(V^\vee\otimes\sO_{\P(V)})/c(\sL)c(\sL^\vee) = 1 + \zeta^2
\end{equation}
so that
\[
H^*(\P_{\P(V)}(\sL^\perp /\sL)) = \frac{\Q[h,\zeta]}{(\zeta^4, h^2 + \zeta^2)}.
\]
Finally, we relate the two descriptions above of the cohomology ring of \(\P_G(\sU)\cong \P_{\P(V)}(\sL^\perp/\sL)\) in the following
\begin{proposition}\label{proposition e local model}
Consider the natural isomorphism $\P_G(\sU) \cong \P_{\P(V)}(\sL^\perp/\sL)$ defined in Proposition \ref{prop whOmega over wtOmega}. The isomorphism induced on their cohomology rings is given by:
\begin{align*}
    \xi &\mapsto \zeta \\
    u_2 &\mapsto \zeta h\\
    u_1 &\mapsto \zeta + h.
\end{align*}
\end{proposition}
\begin{proof}
Clearly we have $\xi \mapsto \zeta$. Moreover, consider the tautological short exact sequence over $\P_G(\sU)$:
$$ 0 \to \sL \to \gamma_\Omega^* \sU \to \sF \to 0.$$
As the composition $\gamma_\Omega^*\sU \subset \sL^\perp \to \sL^\perp/\sL$ vanishes on $\sL$, it factors through $\sF$ and shows $\sF$ as the tautological subbundle of $\sL^\perp / \sL$. Thus, from the above short exact sequence it follows that $u_1 \mapsto h + \zeta$ and $u_2 \mapsto h\zeta$.
\end{proof}

This suffices to determine the cohomology of the \(\Omega\)-varieties in the 6-dimensional variety of O'Grady.
\begin{proposition}\label{prop coho Omegas for OG6} Consider the case \(Y=K\). The cohomology of the \(\Omega\)-varieties is described as follows.
\begin{enumerate}
    \item $H^*(\wt\Omega) =\left( \frac{\Q[u_1]}{(u^4_1)}\right)^{\oplus 256}$.
    \item $H^*(\ol\Omega) =\left( \frac{\Q[\zeta]}{(\zeta^4)}\right)^{\oplus 256}$.
    \item The cohomology of $\wh\Omega$ is an algebra on the cohomology of $\ol\Omega$ and $\wt\Omega$ via pullback: $$\left(\frac{\Q[u_1, \xi]}{(u^4_1, \xi^2 - \xi u_1 + u_1^2/2)}\right)^{\oplus 256} =
    H^*(\wh\Omega)
    = \left(\frac{\Q[\zeta, h]}{(\zeta^4, h^2+\zeta^2)}\right)^{\oplus 256}. $$
    Along this isomorphism we have the identifications $u_1 \mapsto h+ \zeta$ and $\xi \mapsto \zeta.$
\end{enumerate}
\end{proposition}
\begin{proof}
Using Proposition \ref{prop global Omega's 6} the statements follow immediately from the analysis in the local case above.
\end{proof}

For the 10-dimensional variety of O'Grady we need some extra work.

\begin{proposition}\label{prop coho Omegas for OG10} Consider the case \(Y=M\). The cohomology of the \(\Omega\)-varieties  is described as follows.
\begin{enumerate}
\item $H^*(\wt\Omega)
=\frac{H^*(\Omega)[u_1, u_2]}{(-u_1^2+2u_2-c_2(T\Omega), c_4(T\Omega) - u_2^2)}$.
\item $H^*(\ol\Omega) = \frac{H^*(\Omega)[\zeta]}{(\zeta^4 + \zeta^2c_2(T\Omega) +c_4(T\Omega))}$ .
\item $\frac{H^*(\Omega)[u_1, u_2, \xi]}{(-u_1^2+2u_2-c_2(T\Omega), u_2^2-c_4(T\Omega), \xi^2-\xi u_1 + u_2)} =
H^*(\wh\Omega) = 
\frac{H^*(\Omega)[h,\zeta]}{(\zeta^4 + \zeta^2c_2(T\Omega) +c_4(T\Omega), \zeta^2 + h^2 + c_2(T\Omega) )}.$\\
\\
Along this isomorphism we have the following identifications
\begin{align*}
    &\xi \mapsto \zeta,
    & u_1\mapsto \zeta + h,
    && u_2 \mapsto \zeta h.
\end{align*}

\end{enumerate}
\end{proposition}
\begin{proof}
For the first point, in Proposition \ref{prop coho Omegas for OG10} we have seen that $\wt\Omega$ is the relative Lagrangian Grassmannian  of $T\Omega$. Let $\sU$ be the tautological subbundle, then, since it is Lagrangian, we have the short exact sequence
$$0 \to \sU \to T\Omega \to \sU^\vee \to 0.$$
From this one gets the relations for the Chern classes of $\sU$.
To see that these are all the relations between the Chern classes of $\sU$, we consider the Leray spectral sequence for $p_\Omega\colon\wt\Omega \to \Omega$:
$$E_2^{p,q} = H^p(R^q p_*\underline \Q_{\wt\Omega}) \Rightarrow H^{p+q}(\wt \Omega).$$
As $\Omega$ is simply connected, the local system $R^qp_{*}\underline \Q_{\wt\Omega}$ is trivial. Thus the Chern classes of $\sU$ generate the cohomology and we get
$$H^{k}(\wt\Omega) \simeq \bigoplus_{p+q = k}H^p(\Omega)\otimes H^q(G),$$
where $G := LG(V)$ is the Lagrangian Grassmannian of a four dimensional space \(V\).
Finally comparing the dimensions we conclude that the relations we found are the only ones they satisfy.

The other points are now just an application of the projective bundle formula.
\end{proof}

\subsection{The cohomology of the varieties $\Sigma$-varieties}
We  call \(\sL\), \(\sU\), \(\sF\) the universal subbundles on \(\P_\Omega(T\Omega)\simeq \ol\Omega\), \(\sL\sG_\Omega(T\Omega)\simeq \wt\Omega\), \(\P_{\ol\Omega}(\sL^\perp/\sL)\simeq \wh\Omega\) respectively in the case \(Y=M\) and on \(\P_\Omega(V)\simeq\ol\Omega\), \(LG_\Omega(V)\simeq\wt\Omega\), \(\P_{\ol\Omega}(\sL^\perp/\sL) \simeq \wh\Omega\) respectively in the case  \(Y=K\), as introduced in the previous section.
\begin{proposition}\label{restrictions} Let \(Y\) be the variety \(M\) or \(K\).
\begin{enumerate}
\item Under the isomorphism $\wh\Omega \simeq \P_{\wt\Omega}(\sU)$ we have $c_1(\sO_{\wh\Sigma}(\wh\Omega)|_{\wh\Omega})=c_1(\sO_{\wh\Omega}(\wh\Omega_{OG})) = 2c_1(\sL)$.
    \item Under the isomorphism $\wh\Omega \simeq \P_{\ol\Omega}(\sL^\perp/\sL)$ we have $c_1(\sO_{\wh\Omega}(\wh\Sigma)) = 2c_1(\sF) - 2c_1(\sL)$.
    \item Under the isomorphism $\ol\Omega \simeq \P_\Omega(T\Omega)$ for \(Y=M\) and \(\ol\Omega\simeq \P_\Omega(V)\) for \(Y=K\) we have \(c_1(\sO_{\ol\Sigma}(\ol \Omega)|_{\ol\Omega})=2c_1(\sL)\). 
\item Under the isomorphism $\wt\Omega\simeq \sL\sG_\Omega(T\Omega)$ for \(Y=M\) and \(\wt\Omega\simeq LG_\Omega(V)\) for \(Y=K\) we have $ c_1(\sO_{\wt\Omega}(\wt\Sigma)) = 2 c_1(\sU)$.
\end{enumerate}
\end{proposition}
\begin{proof}
We prove the result for \(Y=M\). When \(Y=K\) the same proof applies, and O'Grady's results used in item (1) and (2) hold true, cfr. \cite[\S 2.1]{OG03}.
\begin{enumerate}
      \item The first equality holds by construction, as \(\wh\Omega=\wh\Omega_{OG}\cap \wh\Sigma\). The second equality follows from \cite[(2.0.1) Proposition, item (2)]{OG99} and the fact that the embedding $\wh\Omega \subset \wh\Omega_{OG}$ is identified with $\P_{\wt\Omega}(\sU) \subset \P_{\wt\Omega}(S^2 \sU)$ by construction.
    \item From \cite[Equation~(2.2.1)]{OG99} $\omega_{\wh M} = \sO_{\wh M} (2\wh \Omega_{OG})$. By adjunction for the inclusions $\wh\Omega\subset \wh\Sigma$ and $\wh\Sigma \subset \wh M$ we compute
    \[
    \omega_{\wh\Omega}  = 
    (\omega_{\wh \Sigma} \otimes \sO_{\wh\Sigma}(\wh\Omega))|_{\wh\Omega}=
    \left(\left(\omega_{\wh M} \otimes \sO_{\wh M}(\wh\Sigma)\right)|_{\wh\Sigma} \otimes \sO_{\wh\Sigma}(\wh\Omega) \right)|_{\wh\Omega} = \sO_{\wh\Sigma}(3\wh\Omega_{OG} + \wh\Sigma)|_{\wh\Omega} 
    \]
    hence \(c_1(\sO_{\wh\Omega}(\wh\Sigma))=c_1(\omega_{\wh\Omega})-6c_1(\sL)\) thanks to part  (1).
    On the other hand, \(\wh\Omega\) and \(\ol\Omega\) are projective bundles on \(\ol\Omega\) and \(\Omega\) respectively, hence using the Euler sequence we can express
    \begin{align*}
            \omega_{\wh\Omega} =
    g_\Omega^* \omega_{\ol\Omega} \otimes \det(\sL^\perp/\sL)^\vee \otimes \sF^{\otimes 2} =
    g_\Omega^* \omega_{\ol\Omega} \otimes \det(L^\perp/L)^\vee \otimes \sF^{\otimes 2} =\\
    g_\Omega^* q_\Omega^* \omega_\Omega \otimes \det T\Omega^\vee \otimes \sL^{\otimes 4} \otimes \det(\sL^\perp/\sL)^\vee \otimes \sF^{\otimes 2}= 
    \sL^{\otimes 4} \otimes \sF^{\otimes 2}
    \end{align*}
    where \(\det(\sL^\perp/\sL)\cong \sO_{\wh\Omega}\) has been computed in \eqref{eq c(L^perp/L)}.
Comparing the two writings we get $c_1(\sO_{\wh\Omega}(\wh\Sigma)) = 2c_1(\sF) - 2c_1(\sL)$.
    
    \item Consider the pullback
    \[
    f^*_\Omega c_1(\sO_{\ol\Sigma}(\ol\Omega)|_{\ol\Omega}) = c_1(\sO_{\wh\Sigma}(\wh\Omega)|_{\wh\Omega}) = 2c_1(\sL) 
    \]
    where the last equality is part (1). The claim now follows from the injectivity of $f^*$.
    \item There exists some $\lambda_1,\lambda_2 \in \R$ and some $\omega \in H^2(\Omega)$ such that
    \begin{align*}
        g_\Omega^*(\lambda_1 c_1(\sU) + \omega) &= g^*_\Omega c_1(\sO_{\wt\Omega}(\wt\Sigma))=g^*_\Omega \wt i^*_\Omega c_1(\sO_{\wt M}(\wt\Sigma)) \\
        &= \wh i_\Omega^*\gamma^*c_1(\sO_{\wt M}(\wt \Sigma)) 
        =\wh i^*_\Omega c_1(\sO_{\wh M}(\wh \Sigma + \lambda_2\wh\Omega_{10})) \\
        &=2c_1(\sF)+2(\lambda_2-1)c_1(\sL).
    \end{align*}
    As $H^2(\wh\Omega) = H^2(\Omega) \oplus \Q c_1(\sU) \oplus \Q c_1(\sL)$, and $c_1(\sU) = c_1(\sF) + c_1(\sL)$, we have $\lambda_1 = 2, \omega = 0, \lambda_2=2$. Using the injectivity of \(g_\Omega^*\) we get the claim.
\end{enumerate}
\end{proof}

\begin{remark}\label{rem pullback Sigma tilde}
 We extract from the proof of item (4) here above the following relation, that we are going to use in what follows:
\[
\gamma^*\sO_{\wt M}(\wt\Sigma)=\sO_{\wh M}(\wh\Sigma +2\wh\Omega_{OG}).
\]
\end{remark}

\begin{proposition}\label{prop coho of Sigmas} Let \(Y\) be the variety \(M\) or \(K\). The cohomology of the \(\Sigma\)-varieties is described as follows.
\begin{enumerate}
\item The pullback along the morphism $q\colon\ol\Sigma \to \Sigma$ is injective. Let $\zeta = c_1(\sL)\in H^2(\ol\Omega)$. For any $k$ we have
\[
H^k(\ol \Sigma) = q^*H^k(\Sigma) \oplus \ol i_*q_\Omega^*H^{k-2}(\Omega) \oplus \ol i_* \zeta q_\Omega^*H^{k-4}(\Omega) \oplus \ol i_* \zeta^2 q_\Omega^*H^{k-6}(\Omega).
\]
\item The pullback along the morphism
$f\colon\wh\Sigma \to \ol\Sigma$ is injective, the multiplication $H^{k-2}(\wh \Sigma) \to H^{k}(\wh \Sigma)$ with $c_1(\sO_{\wh\Sigma}(-\wh\Sigma))$ is injective on $f^*H^{k-2}(\ol \Sigma)$ for any $k$. Moreover, for any $k$ we have
\[
H^k(\wh\Sigma) = f^* H^k(\ol \Sigma) \oplus c_1(\sO_{\wh\Sigma}(-\wh\Sigma))\cdot f^*H^{k-2}(\ol \Sigma).
\]
\item The pullback along the morphism $g\colon\wh\Sigma \to \wt\Sigma$ is injective and for any $k$ we have
\[
H^k(\wh\Sigma) = g^* H^k(\wt \Sigma) \oplus \wh i_* g_\Omega^* H^{k-2}(\wt \Omega).
\]
\item The pullback along the morphism $p\colon\wt\Sigma \to \Sigma$ is injective. 
\end{enumerate}
\end{proposition}
\begin{proof}
\begin{enumerate}
    \item Since $\Sigma$ has quotient singularities, the cohomology groups $H^k(\Sigma)$ carry a pure Hodge structure of weight $k$ and the pullback $H^k(\Sigma)\to H^k(\ol\Sigma)$ is injective \cite{Ste76}.
The diagram
\[\xymatrix{
\ol\Sigma \ar[d]^{q}  & \ol \Omega \ar[l]^{\ol i}\ar[d]^{q_\Omega}\\
\Sigma & \Omega \ar[l]^{i}
}
\]
is a hypercubical resolution of $\Sigma$ and its weight spectral sequence, which degenerates at the second page, is supported on the first two columns. As the cohomology of $\Sigma$ carries a pure Hodge structure the differential at the first page must be surjective and we get this short exact sequence
\begin{align}\label{ses-Sigma-Omega}
    0 \to H^k(\Sigma) \xrightarrow{(q^*, i^*)} H^k(\ol\Sigma) \oplus H^k(\Omega) \xrightarrow{\ol i^*- q_\Omega^*} H^k(\ol\Omega) \to 0.
\end{align}

We can split the sequence by defining the following section
\begin{align*}\bigoplus_{r=0}^{3}\zeta^r q^*H^{k-2r}(\Omega) = H^k(\ol \Omega) &\to H^k(\ol\Sigma) \oplus H^k(\Omega)\\
(\zeta^r q^*_\Omega a_r)_r 
&\mapsto \left(\frac{1}{2}\left(\ol i_* \zeta^{r-1} q_\Omega^*a_r\right)_{r\geq 1}, -a_0 \right).
\end{align*}
To see that it is a section, we notice that
\[
 \ol i^*\ol i_*\frac{1}{2} \zeta^{r-1}q_\Omega^*a_r = \frac{1}{2}
[\ol\Omega]|_{\ol\Omega} \cdot  \zeta^{r-1}q_\Omega^*a_r =
\zeta^r q_\Omega^* a_r
\]
where the first equality follows from \cite[Chapter~11, Exercise~1]{Voi07} and the last one from (3) in Proposition \ref{restrictions}.

From the short exact sequence \eqref{ses-Sigma-Omega} we can then extract the desired decomposition.
\item As the morphism $\wh \Sigma \to \ol\Sigma$ is a $\P^1$-fibration and $c_1(\sO_{\wh\Sigma}(-\wh\Sigma))$ pairs with the generic fiber in 2 points by \cite[Proposition (2.3.1)]{OG99}
, the claim follows from the projective bundle formula. 
\item The proof is completely similar to point (1). 
\item Consider the diagram 
\[\xymatrix{
&\wh\Sigma \ar[rd]^g\ar[ld]_f\\
\wt\Sigma \ar[rd]_p&& \ol\Sigma\ar[ld]^q\\
&\Sigma.
}
\]
We have already seen that the pullbacks along $q,g$ and $f$ are injective, hence the pullback along $p$ is too. 
\end{enumerate}
\end{proof}

We conclude describing the object \(E^{0,k}_2\) in the spectral sequence \eqref{eq. spectral sequence}.

\begin{proposition}\label{prop E^0,k_2}
We have 
\begin{align*}
E^{0,k}_2 &\simeq W_k := \{ y\in H^k(\wt Y): \wt j^* m = p^* \sigma' \mathrm{\ for\ some\ } \sigma'\in H^k(\Sigma)\}\\
(y, \sigma, \omega) &\longmapsto y.
\end{align*}
\end{proposition}
\begin{proof}
We write
\[
(y,\sigma,\omega)\in H^k(\wt Y)\oplus H^k(\ol \Sigma)\oplus H^k(\Omega)=E^{0,k}_1
\]
for an element in \(E^{0,k}_1\) and we have the following 
\begin{claim}
$\ker(d_0:E^{0,k}_1\to E^{1,k}_1)$ consist of elements $(y,\sigma,\omega)\in E^{0,k}_1$ such that:
\begin{enumerate}
    \item $\sigma = q^* \sigma'$ for some $\sigma' \in H^k(\Sigma)$ with \(i^*\sigma'=\omega\). 
    \item $\wt j^*y = p^* \sigma'$.
\end{enumerate}
\end{claim}
\begin{proof}[Proof of the claim]
We want to check the three conditions in Lemma \ref{lem E^(0,q)_2}.
\begin{enumerate}
\item For any element $(y,\sigma, \omega)$ of the kernel we have that $\ol i^* \sigma = q^*_\Omega \omega \in H^k(\ol\Omega)$. By Proposition \ref{restrictions} the restriction map $\ol i^* $ respects the decomposition of $H^k(\ol\Sigma)$ of Proposition \ref{prop coho of Sigmas}.(1) and of $H^k(\ol\Omega)$ of Proposition \ref{prop coho Omegas for OG6}.(2) and Proposition \ref{prop coho Omegas for OG10}.(2);  hence the condition above translates to \(\sigma=q^*\sigma'\) for some \(\sigma'\in H^k(\Sigma)\) such that \(i^*\sigma'=\omega\).
\item For any element of the kernel $(y,q^*\sigma', \omega)$, we have that the pullback of $y$ and $q^*\sigma'$ in $H^k(\wh\Sigma)$ coincide, thus
\[
g^*\wt j^*y =f^* q^* \sigma' =  g^*p^* \sigma'
\]
and point (2) follows from the injectivity of $g^*$.
\end{enumerate}
\end{proof}
It follows that any element in \(\ker (d_0:E^{0,k}_1\to E^{1,k}_1)=E^{0,k}_2\) is determined by the choice of \(y\in H^k(\wt Y)\), as \(p^*\) is injective (because \(g^*p^*=f^*q^*\) is injective) and \(\omega=i^*\sigma'\). Such  \(y\in H^k(\wt Y)\) needs to safisfy \(\wt j^*m=p^*\sigma'\) and the claim follows.
\end{proof}

\section{The cohomology of the 10-dimensional singular moduli space}

Assume the notations as in Notation \ref{notation}. In this section we discuss the case of the 10-dimensional singular O'Grady's moduli space \(Y=M\). Using the results of the previous section we can compute the Betti numbers of the \(\Omega\)- and \(\Sigma\)-varieties, that we list in the table below. 

As observed at the beginning of Section \ref{section lehn sorger} the manifold $\Omega$ is a Hyperkähler manifold of K$3^{[2]}$-type and its Betti numbers are know thanks to the G\"ottsche formula \cite{gottsche1990betti}. Thanks to Proposition \ref{prop coho Omegas for OG10} a straighforward computation gives the Betti numbers of the other \(\Omega\)-varieties.

Since the variety $\Sigma$ is isomorphic to double symmetric product of $\Omega$, the rational cohomology of $\Sigma$ is isomorphic to the invariant part of the rational cohomology $\Omega \times \Omega$ \cite[\S III, Theorem~2.4]{bredon-book}, thus we can compute its Betti numbers. Using Proposition \ref{prop coho of Sigmas} we can compute all Betti numbers of the other \(\Sigma\)-varieties. 

Finally, the cohomology of the manifold \(\widetilde M\) has been computed in \cite{dCRS21}.

\begin{equation}\label{tabellina}
\begin{tabu}{ c|c|c|c|c|c|c|c|c|c|c|c } 
  & b_0 & b_2 & b_4 & b_6 & b_8 & b_{10} & b_{12} & b_{14}& b_{16} & b_{18}& b_{20}\\ 
  \hline
 \Omega & 1 & 23 & 276 & 23 & 1 & 0 & 0 & 0 & 0 & 0 & 0\\ 
 \overline{\Omega} & 1 & 24 & 300 & 323 & 323 & 300 & 24 & 1 & 0 & 0 & 0\\ 
 \widetilde{\Omega} & 1 & 24 & 300 & 323 & 323 & 300 & 24 & 1 & 0 & 0 & 0\\
 \widehat{\Omega} & 1 & 25 & 324 & 623 & 646 & 623 & 324 & 25 & 1 & 0 & 0 \\
 \hline
 \Sigma & 1 & 23 & 552 & 6371 & 38756 & 6371 & 552 & 23 & 1 & 0 & 0\\
 \overline{\Sigma} & 1 & 24 & 576 & 6671 & 39078 & 6671 & 576 & 24 & 1 & 0 & 0\\
 \widetilde\Sigma & 1 & 24 & 576 & 6947 & 45426 & 45426 & 6947 & 576 & 24 & 1 & 0 \\
 \widehat \Sigma & 1 & 25 & 600 & 7247 & 45749 & 45749 & 7247 & 600 & 25 & 1 & 0\\
 \hline 
 \widetilde M & 1 & 24 & 300 & 2899 & 22150 & 126156 & 22150 & 2899 & 300 & 24 & 1 \\ \hline
\end{tabu}
\end{equation}

\begin{proposition}\label{prop Sigma and Omega} We have, for any \(k\ge 0\):
\begin{enumerate}
    \item \(i^*:H^k(\Sigma)\rightarrow H^k(\Omega)\) is surjective.
    \item \(\overline i^*:H^k(\overline\Sigma)\rightarrow H^k(\overline\Omega)\) is surjective.
    \item \(\widehat i^*:H^k(\widehat \Sigma)\rightarrow H^k(\widehat\Omega)\) is surjective.
    \item \(j^*:H^2(M)\to H^2(\Sigma)\) is an isomorphism.
    \item \(\wt i_M^*:H^k(\wt M)\rightarrow H^k(\wt\Omega)\) is surjective.
\end{enumerate}
\end{proposition}
\begin{remark}\label{remark:reflexive}
The variety $M$ has a symplectic form on its regular part, in other words a reflexive form $\sigma_M \in H^0( \Omega_M^{[2]}, M)$, where $\Omega_M^{[2]} = (\Omega^2_{M})^{\vee \vee} = (i^{reg}_{*}\Omega^2_{M^{reg}})$. This induces a nontrivial class in $H^2(M, \C)$. An analogous statement holds for $\Sigma$.
\end{remark}
\begin{proof} Since the involved varieties have no odd cohomology (cfr. \eqref{tabellina}) we only need to prove the statement for their even cohomology groups.
\begin{enumerate}
    \item For \(k=0\) and \(k>8\) the statement is obvious. For \(k=2\): given the transcendental lattice \(T(\Sigma)\subset H^2(\Sigma,\Z)\), the intersection \(\ker(i^*)\cap T(\Sigma)\subseteq T(\Sigma)\) is a Hodge substructure. The reflexive symplectic form of \(\Sigma\) restricts to the symplectic form of \(\Omega\) (\cite[Theorem~2.3]{Kal06b}), hence the orthogonal complement of \(\ker(i^*)\cap T(\Sigma)\) in \(T(\Sigma)\) contains the symplectic form of \(\Sigma\) and then it coincides with \(T(\Sigma)\) by minimality of the trascendental lattice. It follows that \(i^*\) is injective on \(T(\Sigma)\).  Taking a general locally trivial deformation \(\Sigma_t\) of \(\Sigma\) and defining \(\Omega_t\) its singular locus, for the very general \(\Sigma_t\) we have that \(H^2(\Sigma_t,\Z)=T(\Sigma_t)\), hence the pullback of the inclusion \(i_t:\Omega_t\hookrightarrow \Sigma_t\) is injective on \(H^2(\Sigma_t,\Z)\); we conclude that \(i^*:H^2(\Sigma)\rightarrow H^2(\Omega)\) is injective, hence an isomorphism for dimensional reasons, cfr. \eqref{tabellina}. For the case \(k=4\) we consider the commutative diagram:
    \[
    \xymatrix{
    \Sym^2 H^2(\Sigma)\ar[r]^{\Sym^2i^*}\ar[d] & \Sym^2 H^2(\Omega)\ar[d]\\
    H^4(\Sigma)\ar[r]^{i^*} & H^4(\Omega).
    }
    \]
    The left and right vertical arrow are injective by \cite[Proposition 5.16]{BL18} and \cite[Theorem 1.5]{Ver96} respectively, hence the right one is an isomorphism for dimensional reasons and \(i^*\) is surjective. We are left with the cases \(k=6,8\). For the case \(k=6\), let \(L:H^k(\Sigma)\to H^{k+2}(\Sigma)\) be the Lefschetz operator and \(L_\Omega\) its restriction to \(\Omega\). We have the commutative square: 
    \[
    \xymatrix{
     H^{2}(\Sigma)\ar[r]^{i^*}\ar[d]^{L^2} & H^{2}(\Omega )\ar[d]^{L^2_\Omega}\\
    H^6(\Sigma)\ar[r]^{i^*} & H^6(\Omega)
    }
    \]
    The right vertical arrow is an isomorphism by Hard Lefschetz, hence the surjectivity of \(i^*\) in degree 6 follows from the surjectivity of \(i^*\) in degree 2. The case \(k=8\) is proved with the same argument, starting from the surjectivity of \(i^*:H^0(\Sigma)\rightarrow H^0(\Omega)\).
\item  Following Proposition \ref{prop coho of Sigmas} and Proposition \ref{prop coho Omegas for OG10}: 
\begin{align*}
H^k(\overline \Sigma)= q^* H^k(\Sigma)\oplus \overline{i}_*\bigl[q^*_\Omega H^{k-2}(\Omega)\oplus \zeta\cdot q^*_\Omega H^{k-4}(\Omega) \oplus \zeta^2\cdot q^*_{\Omega}H^{k-6}(\Omega)\bigl],\\
H^k(\overline \Omega)= q^*_\Omega H^k(\Omega)\oplus  \zeta\cdot q^*_\Omega H^{k-2}(\Omega) \oplus  \zeta^2\cdot q^*_\Omega H^{k-4}(\Omega) \oplus \zeta^3\cdot q^*_\Omega H^{k-6}(\Omega).
\end{align*}
We look at the morphism \(\overline i^*:H^k(\overline\Sigma)\rightarrow H^k(\overline\Omega)\) on the factors of the above decomposition. From the surjectivity of \(i^*\) proved in (1) it follows:
\[
\overline i^*q^*H^k(\Sigma)= q^*_\Omega i^* H^k(\Sigma)=q_\Omega^* H^k(\Omega).
\]
Furthermore, from Proposition \ref{restrictions}:
\[
\overline i^*\overline i_* q^*_\Omega H^l(\Omega) =[\ol\Omega]|_{\ol\Omega}\cdot q^*_\Omega H^l(\Omega)=2\zeta\cdot  q^*_\Omega H^l(\Omega)
\]
hence \(\overline i^*\) is surjective for any \(k\ge 0\).
\item Following Proposition \ref{prop coho Omegas for OG10}.(3) and Proposition \ref{prop coho of Sigmas}.(2) we have
\begin{align*}
    H^k(\wh \Sigma)=f^* H^k(\ol\Sigma)\oplus c_1(\sO_{\wh\Sigma}(-\wh\Sigma))\cdot f^* H^{k-2}(\ol\Sigma)
\end{align*}
\[
 H^k(\wh\Omega)=f^*_\Omega H^k(\ol \Omega)\oplus h\cdot f^*_\Omega H^{k-2}(\ol\Omega).
\]
Using the decompositions above: from the surjectivity in part (2) we have that \(\wh i^*|_{f^* H^k(\ol \Sigma)}:f^* H^k(\ol \Sigma)\twoheadrightarrow f^*_\Omega \ol i^* H^k(\ol\Sigma)=f^*_\Omega H^k(\ol\Omega)\), hence the claim follows from \ref{restrictions}.(2).
\item The injectivity of \(j^*\) is proven exactly as the injectivity of \(i^*:H^2(\Sigma)\to H^2(\Omega)\) in item (1), hence we conclude for dimensional reasons, cfr. \cite[Theorem 1.7]{PR13} and \eqref{tabellina}.
\item As $\Omega$ is of $K3^{[2]}$-type, it is well-known that $H^*(\Omega) = H^0(\Omega)\langle H^2(\Omega) \rangle$, cfr. \cite[Corollary~3.2]{GKLR}. Proposition \ref{prop coho Omegas for OG10} shows that 
\[H^*(\wt\Omega) = H^0(\wt\Omega)\langle p^*_\Omega H^2(\Omega), u_1 \rangle = \langle H^2(\wt \Omega)\rangle.\] 
We have
\[\wt i^*_M H^2(\wt M) = \wt i^*_M (\pi^* H^2(M)\oplus c_1(\sO_{\wt M}(\wt\Sigma))\cdot \Q) = p_\Omega^*H^2(\Omega) \oplus u_1\cdot \Q = H^2(\wt\Omega)
\] 
where the second equality follows from the surjectivity of \(i^*_M:H^2(M)\to H^2(\Omega)\), obtained as combination of item (1) and item (4), and Proposition~\ref{restrictions}.(4). We
conclude that the pullback $\wt i^*_M$ is surjective in any degree.
\end{enumerate}
\end{proof}

\begin{proposition}\label{odd-bettis}
We have 
\begin{align*}
&H^{2k}(M)\simeq E_2^{0,2k} &\mathrm{\ and \ }
&&H^{2k+1}(M) \simeq E_2^{1,2k}.
\end{align*}
In particular, if non-zero the groups \(H^{2k}(M)\) and \(H^{2k+1}(M)\) carry a pure Hodge structure of weight \(2k\). 
\end{proposition}
\begin{proof}
At first, notice that \(E_2^{2,k}=0\) for any \(k\in\mathbb Z\): the differential \(d_1:E_1^{1,k}\to E^{2,k}_1\) is surjective by Lemma \ref{lem E_2^{1-2,q}} and Proposition \ref{prop Sigma and Omega}.(3). Hence the statement follows from the fact that all varieties appearing in the semi-simplicial resolution of $M$ have trivial cohomology in odd degree.
\end{proof}
\begin{corollary}\label{cor pullback inj}
The pullback $\pi^*\colon  H^{2k}(M) \to H^{2k}(\wt M)$ is injective for any $k$.
\end{corollary}
\begin{proof}
Thanks to the above proposition, the morphism
\begin{equation} \label{1st map for M}
     \pi^*  \oplus q^*i^*\oplus i_M^*\colon H^{2k}(M) \to H^{2k}(\wt M)\oplus H^{2k}(\ol \Sigma)\oplus H^{2k}(\Omega)
\end{equation}
is an isomorphism on its image $E_2^{0, 2k}$. By Proposition \ref{prop E^0,k_2} the restriction of the projection onto the cohomology of $\wt M$ 
\begin{equation}\label{2nd map for M}
    E_2^{0,2k}\subset H^{2k}(\wt M)\oplus H^{2k}(\ol \Sigma)\oplus H^{2k}(\Omega) \to H^{2k}(\wt M)
    \end{equation}
is an isomorphism onto its image. Taking the composition of \eqref{1st map for M} and \eqref{2nd map for M} we get the claim.  \end{proof}

\begin{proposition}\label{prop Betti estimates 1}
We denote \(e^{i,j}_k=\dim E^{i,j}_k\). We have 
\begin{align*}
b_{2k}(M) &\leq b_{2k}(\wt M) - (b_{2k} (\wt\Omega) - b_{2k}(\Omega)),\\
b_{2k}(M)&\geq e_1^{0,2k}-e_1^{1,2k} + e_1^{2,2k} \mathrm{\ and}\\
b_{2k +1}(M) &= e_1^{1,2k} - e_1^{2,2k}  - e_1^{0,2k}+ b_{2k}(M).
\end{align*}
\end{proposition}
\begin{proof}
Recall the definition of $W_{2k}$ from Proposition \ref{prop E^0,k_2} and that $\dim W_{2k} = b_{2k}(M)$ from Proposition \ref{prop E^0,k_2}.
We estimate the dimension of $W_{2k}$. As 
$H^{2k} (\wt M) \to H ^{2k}(\wt\Omega)$ is surjective from item (4) of Proposition \ref{prop Sigma and Omega} and under this map $W_k$ maps onto $p^*_\Omega H^{2k} (\Omega)$ we get a surjective map on the quotients
\[
H^{2k} (\wt M) / W_{2k} \to H^{2k}(\wt\Omega)/p^*_\Omega H^{2k}(\Omega)
\]
whence the first claimed inequality.\\
For the second inequality, we look at the complex 
\[
E_1^{0,2k}\xrightarrow{d_0} E_1^{1,2k} \xrightarrow{d_1} E_1^{2,2k} \to 0
\]
then using Proposition \ref{odd-bettis} we obtain
$$b_{2k}(M) =\dim \ker d_0 = e_1^{0,2k} - \dim \Im d_0 \geq e_1^{0,2k} - \dim\ker d_1 = e_1^{0,2k}-e_1^{1,2k} + e_1^{2,2k}$$
where the last equality follows from the surjectivity of $d_1$, cfr. Lemma \ref{lem E_2^{1-2,q}} and Proposition \ref{prop Sigma and Omega}.(3).\\
For the last inequality, as $H^{2k+1}(M) \simeq \ker d_1/\Im d_0$  (cfr. Proposition \ref{odd-bettis}) we have
$$b_{2k+1}(M) = \dim \ker d_1 - \dim\Im d_0 =  e_1^{1,2k} - e_1^{2,2k} - e_1^{0, 2k} + b_{2k}(M) $$
where \(\dim \ker d_1=e^{1,2k}_1-e^{2,2k}_1\) follows from the surjectivity of \(d_1\), cfr. Lemma \ref{lem E_2^{1-2,q}} and Proposition \ref{prop Sigma and Omega}.(3).
\end{proof}
\begin{corollary}\label{cor Euler char of M}
The Euler characteristic of $M$ is \(\chi(M)=123606\).
\end{corollary}
\begin{proof} Using the relation \(b_{2k+1}(M)=e^{1,2k}_1-e^{2,2k}_1-e_1^{0,2k}+b_{2k}(M)\) computed in Proposition \ref{prop Betti estimates 1} we get
\begin{align*}
    \chi(M) &= \sum_{k=0}^{10} (b_{2k}(M) - b_{2k+1}(M)) = \sum_{k=0}^{10} ( e_1^{0,2k} - e_1^{1,2k}  + e_1^{2, 2k}) \\
    &=  \chi(\wt M) + \chi(\Omega) + \chi(\ol\Sigma)
    -\chi(\wh\Sigma) - \chi(\wt\Omega) - \chi(\ol\Omega)  + \chi(\wh\Omega) \\
    &= 123606
\end{align*}
where the last equality is a straightforward computation with the dimensions in \eqref{tabellina}.
\end{proof}
\begin{remark}
For \(2k\le 10\) we have an injection \(\Sym^k(M)\hookrightarrow H^{2k}(M)\) by \cite[Proposition 5.16]{BL18}, which gives \( b_{2k}(M) \geq \binom{22 + k}{k}\).
Nevertheless, this estimate turns out to be weaker than (or equal to) the one in Proposition \ref{prop Betti estimates 1}.
\end{remark}

\begin{corollary}\label{cor Betti estimates 1} We call \(b_i:=b_i(M)\). We have
\[
    \begin{tabu}{c|c|c}
         b_0=1 & b_1 = 0 & b_2 = 23 \\ \hline 
b_3=0 & b_4 = 276 & b_5 = 0 \\ \hline 
2323 \leq b_6 \leq 2599 & b_7 \leq 276 & 15480 \leq b_8 \leq 21828\\ \hline 
b_9 \leq 6348 & 87101 \leq b_{10} \leq 125856 & b_{11} \leq 38755 \\  \hline 
15755 \leq b_{12} \leq 22126 & b_{13} \leq 6371 & 2346 \leq b_{14} \leq 2898 \\ \hline 
b_{15} \leq 552 & 277 \leq b_{16} \leq 300 & b_{17} \leq 23\\ \hline 
23 \leq b_{18} \leq 24 & b_{19} \leq 1   & b_{20} = 1.
    \end{tabu}
    \]
\end{corollary}
\begin{proof}
This is a direct computation of the bounds in Proposition \ref{prop Betti estimates 1} using the dimensions in \eqref{tabellina}. The estimates on odd Betti numbers are given by the following formula:
$$b_{2k +1}(M) \leq e_1^{1,2k} - e_1^{2,2k}  - e_1^{0,2k}+b_{2k}(\wt M) - (b_{2k} (\wt\Omega) - b_{2k}(\Omega))$$
which is straightforward from Proposition \ref{prop Betti estimates 1}.
\end{proof}
\begin{remark}
The 1st and 2nd Betti numbers were already known \cite[Theorem 1.7]{PR13} and \cite[Lemma~2.1]{BL20}.
\end{remark}
\begin{proposition}\label{prop Betti estimates 2} We have
\begin{align*}
    g^*p^*\Sym^{9-k}H^2(\Sigma) \subseteq \Im \{d_0\colon E^{0,2k}_1 \to E^{1,2k}_1\} \ \ \ & \mathrm{for}\ \  10\le 2k \le 12, \\
    g^*\Sym^{9-k}H^2(\wt\Sigma) \subseteq \Im \{d_0\colon E^{0,2k}_1 \to E^{1,2k}_1\} \ \ \ & \mathrm{for}\ \ 14\le 2k \le 18.
\end{align*}
As consequence, if we denote \(e^{i,j}_k=\dim E^{i,j}_k\) then we have 
\begin{align*}
\mathrm{\ for}  \ 10\le 2k\le 12:\ &b_{2k}(M)\le e_1^{0,2k}-\binom{31-k}{9-k} \\
&b_{2k+1}(M)\le e_1^{1,2k}-e_1^{2,2k}- \binom{31-k}{9-k}   \\
\mathrm{\ for}  \ 14\le 2k\le 18:&\ b_{2k}(M)\le e_1^{0,2k}-\binom{32-k}{9-k} \\
&b_{2k+1}(M)\le e_1^{1,2k}-e_1^{2,2k}- \binom{32-k}{9-k}.
\end{align*}
\end{proposition}
\begin{proof}
Assume \(10\le 2k\le 18\). From the very definition of the spectral sequence \eqref{eq. spectral sequence} we have \((\wt j\circ g)^* H^{2k}(\wt M)\subseteq \Im \{d_0:E_1^{0,2k}\to E^{1,2k}_1\}\). Consider the Lefschetz operator \(L:H^k(\wt M)\to H^{k+2}(\wt M)\) and its restriction \(L_{\wt\Sigma}\) to \(\wt \Sigma\). We have a commutative square
\[\xymatrix{
H^{2k}(\wt M) \ar[r]^{\wt j^*} & H^{2k}(\wt \Sigma) \\
H^{2(10-k)}(\wt M) \ar[u]^{L^{2k-10}}_{\simeq}\ar[r]^{\wt j^*} & H^{2(10-k)}(\wt \Sigma)\ar[u]^{L_{\wt\Sigma}^{2k-10}}
\\
H^{2(9-k)}(\wt M) \ar@{^{(}->}[u]^L\ar[r]^{\wt j^*} & H^{2(9-k)}(\wt \Sigma)\ar[u]^L \ar@/_3pc/[uu]_{\simeq}.
}
\]
where the isomorphism  $H^{2(9-k)}(\wt\Sigma) \simeq H^{2k}(\wt \Sigma)$ is given by Hard Lefschetz for intersection cohomology \cite[\S1.4]{dCM-survey}. Furthermore we consider the following commutative diagram
\[
\xymatrix{
H^{2(9-k)}(\wt M) \ar[r]^{\wt j^*} & H^{2(9-k)}(\wt\Sigma) \\
H^{2(9-k)}(M) \ar[r]^{j^*} \ar@{^{(}->}[u]^{\pi^*} & H^{2(9-k)}(\Sigma) \ar@{^{(}->}[u]^{p^*} \\
\Sym^{9-k}H^2(M) \ar@{^{(}->}[u] \ar[r]^{\Sym j^*}_{\sim} & \Sym^{9-k}H^2(\Sigma) \ar@{^{(}->}[u]
}
\]
where:  \(\pi^*\) is injective by Corollary \ref{cor pullback inj}, \(p^*\) is injective by Proposition \ref{prop coho of Sigmas}.(4), the vertical maps from the symmetric products to the respective cohomology groups are injective because by assumption \(9-k\le 4\), see \cite[Proposition 5.16]{BL18}, and \(\Sym^{9-k}j^*\) is an isomorphism by Proposition \ref{prop Sigma and Omega}.(4). We conclude that \(p^*\Sym^{9-k}H^2(\Sigma)\subseteq \wt j^*H^{2k}(\wt M)\), hence the statement.

When \(14\le 2k\le 18\) we can improve the result as follows. Consider the following commutative diagram
\begin{align}\label{diagram}
 \xymatrix{
H^{2(9-k)}(\wt M) \ar[r]^{\wt j^*} & H^{2(9-k)}(\wt\Sigma)\ar[r]^{\wt i^*} & H^{2(9-k)}(\wt\Omega) \\
\Sym^{9-k} H^2(\wt M) \ar@{^{(}->}[u] \ar[r]^{\Sym \wt j^*}_{\sim} & \Sym^{9-k} H^2(\wt\Sigma) \ar[u] \ar[r]^{\Sym \wt i^*}_{\sim} & \Sym^{9-k}H^2(\wt \Omega) \ar[u] 
}
\end{align}
where the left vertical arrow is injective because by hypothesis \(9-k\le 2\), see  \cite[Theorem 1.5]{Ver96}, and the lower orizontal arrows are isomorphisms by Proposition \ref{prop Sigma and Omega}.(5) and dimensional reasons. Using the description of \(H^2(\wt\Omega)\) in Proposition \ref{prop coho Omegas for OG10}.(1) we get
\[\Sym^{9-k}H^2(\wt\Omega)= \Sym^{9-k}p^*_\Omega H^2(\Omega)\bigoplus_{l=1}^{9-k}u_1^l\cdot \Sym^{9-k-l}p_\Omega^*H^2(\Omega)
\]
where we are assuming as convention \(p^*_\Omega\Sym^0H^2(\Omega)=p^*_\Omega H^0(\Omega)\); note that in the direct sum above \(l\le 2\). Using the injectivity of \(\Sym^m H^2 (\Omega)\to H^{2m}(\Omega)\) for any \(m\le 2\), cfr. again \cite[Theorem 1.5]{Ver96}, and the decomposition
\[
H^{2(9-k)}(\wt\Omega)=p^*_\Omega H^{2(9-k)}(\Omega)\bigoplus_{l=1}^{9-k} u_1^l\cdot p^*_\Omega H^{2(9-k-l)}(\Omega)
\]
given by Proposition \ref{prop coho Omegas for OG10}.(1) we obtain that the right vertical arrow in the diagram~\eqref{diagram} is injective, hence the same holds true for the right central arrow of the diagram.
We conclude that \(\Sym^{9-k}H^2(\wt\Sigma)\subseteq \wt j^* H^{2k}(\wt M)\) hence the first statement. 

The claim on the Betti numbers is obtained as in the proof of Proposition \ref{prop Betti estimates 1}, using again the injectivity of \(p^*\) and the one of \(g^*\), cfr. Proposition \ref{prop coho of Sigmas}.
\end{proof}
\begin{remark}\label{Remark LLV}
Notice that crucial in the argument of Proposition \ref{prop Betti estimates 2} is the study of the kernel $H^2(\wt M) \to H^2(\wt \Sigma)$, which is a Hodge substructure of $H^2(\wt M)$. Now, the cohomology ring $H^*(\wt M)$ decomposes in irreducible representations for its LLV algebra and in the argument we have analysed the pullback just on the Verbitsky component (cfr. Remark \ref{rem introduction LLV}); it is plausible that analysing what happens on the other irreducible subrepresentations one may completely determine the cohomology of $M$. We intend to do so in a future work.
\end{remark}
\begin{corollary}\label{cor Betti estimates 2}
We call \(b_i:=b_i(M)\). We have
\[
\begin{tabu}{c|c|c|c|c}
b_{10} \le 117877 & b_{11} \le 30776 & b_{12} \le 20426 & b_{13} \le 4671 & b_{14} \le 2623 \\ \hline
b_{15}\le 277 & b_{16}= 277 &  b_{17}= 0 & b_{18} = 23 & b_{19}= 0
\end{tabu}
\]
\end{corollary}
\begin{proof}
The estimates on \(b_{12}\), \(b_{13}\), \(b_{14}\) and \(b_{15}\) follow from Proposition \ref{prop Betti estimates 2} using the dimensions computed in table \eqref{tabellina}. The Betti numbers \(b_{16}\), \(b_{17}\), \(b_{18}\) and \(b_{19}\) are obtained from the upper bound in Proposition \ref{prop Betti estimates 2} and the lower bound in Corollary \ref{cor Betti estimates 1}.
\end{proof}

Corollary \ref{cor pullback inj}, Corollary \ref{cor Euler char of M}, Corollary \ref{cor Betti estimates 1} and Corollary \ref{cor Betti estimates 2} prove Theorem \ref{thm Betti of M} stated in the introduction.

\section{The cohomology of the 6-dimensional singular moduli space}

Assume the notations as in Notation \ref{notation}. In this section we discuss the case of the 6-dimensional singular O'Grady's variety \(Y=K\). Using the results in Section \ref{section cohomology of Sigma and Omega} we can compute the Betti numbers of all varieties involved in the spectral sequence \eqref{eq. spectral sequence}, that we list in the table below.

As observed at the beginning of Section \ref{section introduction}, the variety $\Omega$ consists of 256 points and using Proposition \ref{prop coho Omegas for OG6} we get the Betti numbers of the \(\Omega\)-varieties.  

The variety $\Sigma$ is isomorphic to $(A\times A^\vee)/ \pm1$. For the abelian $4$-fold $A\times A^\vee$ we have
\begin{align*}
H^1(A\times A^\vee, \Z) = \Z ^8 && H^k(A\times A^\vee, \Z) \simeq \Lambda^k H^1(A\times A^\vee, \Z)
\end{align*}
so that 
$$
b_k(A\times A^\vee) = \binom{8}{k}.
$$
Finally, analysing the action of $\pm 1$ on forms, we compute 
$$
H^k(\Sigma, \mathbb Z) = H^k(A \times A^\vee, \mathbb Z)^{(\pm 1)^*} =
\begin{cases}
0 &\text{ for $k$ odd} \\
H^k(A \times A^\vee, \mathbb Z) &\text{for  $k$ even.} 
\end{cases}
$$
The Betti numbers of the other \(\Sigma\)-varieties follow from Proposition \ref{prop coho of Sigmas}.

Finally, the cohomology of the manifold \(\widetilde K\) has been computed in \cite{MRS18}.

\begin{equation}\label{tabellina OG6}
\begin{tabu}{ c|c|c|c|c|c|c|c } 
 
  & b_0 & b_2 & b_4 & b_6 & b_8 & b_{10} & b_{12} \\ 
  \hline
 \Omega & 256 & 0 & 0 & 0 & 0 & 0 & 0\\ 
 \overline{\Omega} & 256 & 256 & 256 & 256 & 0 & 0 & 0 \\ 
 \widetilde{\Omega}& 256 & 256 & 256 & 256 & 0 & 0 & 0 \\
 \widehat{\Omega} & 256 & 512 & 512 & 512 & 256 & 0 & 0 \\
 \hline
 \Sigma & 1 & 28 & 70 & 28& 1& 0 & 0\\
 \overline{\Sigma} & 1 & 284 & 326 & 284 & 1 & 0 & 0\\
 \widetilde\Sigma & 1 & 29 & 354 & 354 & 29 & 1 & 0 \\
 \widehat \Sigma & 1 & 285 & 610 & 610 & 285 & 1 & 0 \\
 \hline 
 \widetilde K & 1 & 8 & 199 & 1504 & 199 & 8 & 1  \\ \hline
\end{tabu}
\end{equation}

\begin{proposition}\label{prop surj of hat i, OG6} We have: 
\begin{enumerate}
    \item \(\ol i^*:H^k(\ol\Sigma)\to H^k(\ol\Omega)\) is surjective for any \(k\ge 2\).
    \item \(\wh i^*:H^k(\wh\Sigma)\to H^k(\wh\Omega)\) is surjective for any \(k\ge 4\).
    \item \(j^*:H^2(K)\to H^2(\Sigma)\) is injective.
\end{enumerate}
\end{proposition}
\begin{proof}
Since the \(\Omega\)- and the \(\Sigma\)-varieties have no odd cohomolgy (cfr. \eqref{tabellina OG6}) we only need to prove the first two statements for their even cohomolgy groups.
\begin{enumerate}
    \item Following Proposition \ref{prop coho of Sigmas}.(1) and Proposition \ref{prop coho Omegas for OG6}.(2) we have:
    \[
    H^k(\ol\Sigma)=q^*H^k(\Sigma)\oplus \ol i_*q^*_\Omega H^{k-2}(\Omega)\oplus\ol i_*\zeta q^*_\Omega H^{k-4}(\Omega)\oplus \ol i_*\zeta^2 q^*_\Omega H^{k-6}(\Omega)
    \]
    \[
    H^k(\ol\Omega)=q^*_\Omega H^k(\Omega)\oplus \zeta q^*_\Omega H^{k-2}(\Omega)\oplus \zeta^2 q^*_\Omega H^{k-4}(\Omega)\oplus \zeta^3 q^*_\Omega H^{k-6}(\Omega)
    \]
    For any \(a\in H^l(\Omega)\) we have \(\ol i^*\ol i_* q^*_\Omega a=[\ol \Omega]|_{\ol\Omega}\cdot q^*_\Omega a=2\zeta q^*_\Omega a\), where the last equality is Proposition \ref{restrictions}.(3). It follows that \(\ol i^*\) respects the decompositions above, and the surjectivity follows from \(H^k(\Omega)=0\) for any \(k\ge 2\).
\item Following Proposition \ref{prop coho Omegas for OG6}.(3) and Proposition \ref{prop coho of Sigmas}.(2) we have:
\begin{align*}
    H^k(\wh \Sigma)=f^* H^k(\ol\Sigma)\oplus c_1(\sO_{\wh\Sigma}(-\wh\Sigma))\cdot f^* H^{k-2}(\ol\Sigma)
\end{align*}
\[
 H^k(\wh\Omega)=f^*_\Omega H^k(\ol \Omega)\oplus h\cdot f^*_\Omega H^{k-2}(\ol\Omega).
\]
Using the decompositions above: from the statement in part (1) we have that \(\wh i^*|_{f^* H^l(\ol \Sigma)}:f^* H^l(\ol \Sigma)\twoheadrightarrow f^*_\Omega \ol i^* H^l(\ol\Sigma)=f^*_\Omega H^l(\ol\Omega)\) for any \(l\ge 2\), hence the claim follows from \ref{restrictions}.(2). 
\item The statement is proven exactly as the injectivity of \(H^2(\Sigma)\to H^2(\Omega)\) in the \(M\) case in Proposition \ref{prop Sigma and Omega}.(1): applying \cite[Theorem 2.3]{Kal06b} we have that the restriction to \(\Sigma\) of the reflexive symplectic form \(\sigma_K\) of \(K\) is a reflexive symplectic form on \(\Sigma\)\footnote{The observations of Remark \ref{remark:reflexive} apply for the varieties $K$ and $\Sigma\subset K$ too.}, hence \(\sigma_K\notin \ker j^*\) and we can conclude also in this case by taking a general locally trivial deformation of \(K\).
\end{enumerate}
\end{proof}

\begin{proposition}\label{odd-bettis OG6}
We have 
\begin{align*}
&H^{2k}(K)\simeq E_2^{0,2k} &\mathrm{\ and \ }
&&H^{2k+1}(K) \simeq E_2^{1,2k}.
\end{align*}
In particular, if non-zero the groups \(H^{2k}(K)\) and \(H^{2k+1}(K)\) carry a pure Hodge structure of weight \(2k\). 
\end{proposition}
\begin{proof}
Observe that it is enough to prove that \(E^{2,2k}_2=0\) for any \(k\in\mathbb Z\), since all varieties appearing in the spectral sequence have trivial cohomology in odd degrees. The statement is non-trivial only for \(0\le k\le 8\), since otherwise \(E^{2,k}_1=H^k(\wh\Omega)=0\). Since by definition \(E^{3,k}_1=0\) for any \(k\), the statement follows once proved that the differential \(d_1:E^{1,k}_1\to E^{2,k}_1\) is surjective. Because of the description of \(d_1\) given in \eqref{eq second differentials}, the statement follows immediately from Proposition \ref{prop surj of hat i, OG6}.(2) when \(k\neq 2\). For the case \(k=2\), we argue as follows. We look at the following pullback appearing in the definition of \(d_1\) (cfr. \eqref{eq second differentials}):
\[
g^*_\Omega+f^*_\Omega:H^2(\wt \Omega)\oplus H^2(\ol \Omega)\to H^2(\wh\Omega).
\]
We show that \(d_1\) is surjective proving that the morphism above is surjective. Because of the description of the square of \(\Omega\)-varieties given in Proposition \ref{prop global Omega's 6}, the morphism above reads as
\[
g_\Omega^*-f_\Omega^*:H^2(LG_\Omega(V))\oplus H^2(\P_\Omega(V))\to H^2(\P_{\ol\Omega}(\sU))
\]
and its surjectivity follows from the double decomposition of the cohomology ring of \(\P_{\wt\Omega}(\sU)\cong \P_{\ol\Omega}(\sL^\perp/\sL)\), cfr. Proposition \ref{proposition e local model}.
\end{proof}
\begin{corollary}\label{cor pullback inj OG6}
The pullback $\pi^*\colon  H^{2k}(K) \to H^{2k}(\wt K)$ is injective for any $k$.
\end{corollary}
\begin{proof}
Thanks to the above proposition, the morphism
\begin{equation} \label{1st map}
     \pi^* \oplus q^*i^*\oplus i_K^*\colon H^{2k}(K) \to H^{2k}(\wt K)\oplus H^{2k}(\ol \Sigma)\oplus H^{2k}(\Omega)
\end{equation}
is an isomorphism on its image $E_2^{2k}$. By Proposition \ref{prop E^0,k_2} the restriction of the projection onto the cohomology of $\wt K$ 
\begin{equation}\label{2nd map}
    E_2^{2k}\subset H^{2k}(\wt K)\oplus H^{2k}(\ol \Sigma)\oplus H^{2k}(\Omega) \to H^{2k}(\wt K)
    \end{equation}
is an isomorphism on its image. Taking the composition of \eqref{1st map} and \eqref{2nd map} we get the claim.  \end{proof}

\begin{proposition}\label{prop Betti estimates 1 OG6}
We denote \(e^{i,j}_k=\dim E^{i,j}_k\). We have 
\begin{align*}
b_{2k}(K) &\leq b_{2k}(\wt K), \\
b_{2k}(K)&\geq e_1^{0,2k}-e_1^{1,2k} + e_1^{2,2k} \mathrm{\ and}\\
b_{2k +1}(K) &= e_1^{1,2k} - e_1^{2,2k}  - e_1^{0,2k}+ b_{2k}(K).
\end{align*}
\end{proposition}
\begin{proof}
The first inequality follows from the very definition of \(W_{2k}\) in Proposition \ref{prop E^0,k_2} and from \(b_{2k}(K)=\dim W_{2k}\), see Proposition \ref{odd-bettis OG6}.
For the second equality, consider the sequence
\[
E_1^{0,2k}\xrightarrow{d_0} E_1^{1,2k}\xrightarrow{d_1} E^{2,2k}_1.
\]
Using again Proposition \ref{odd-bettis OG6} we have
\[
b_{2k}(K)=\dim \ker d_0=e_1^{0,2k}-\dim\Im d_0\ge e^{0,2k}_1-\dim\ker d_1=e^{0,2k}_1-e^{1,2k}_1+e^{2,2k}_1
\]
where the last equality follows from the surjectivity of \(d_1\) proven in the proof of Proposition \ref{odd-bettis OG6}.\\
For the last inequality, as $H^{2k+1}(K) \simeq \ker d_1/\Im d_0$  (cfr. Proposition \ref{odd-bettis OG6}) we have
$$b_{2k+1}(K) = \dim \ker d_1 - \dim\Im d_0 =  e_1^{1,2k} - e_1^{2,2k} - e_1^{0, 2k} + b_{2k}(M) $$
where \(\dim \ker d_1=e^{1,2k}_1-e^{2,2k}_1\) follows again from the surjectivity of \(d_1\).
\end{proof}

\begin{corollary}\label{cor Euler char of K}
The Euler characteristic of $K$ is \(\chi(K)=1208\).
\end{corollary}
\begin{proof} Using the relation \(b_{2k+1}(K)=e^{1,2k}_1-e^{2,2k}_1-e_1^{0,2k}+b_{2k}(K)\) computed in Proposition \ref{prop Betti estimates 1 OG6} we get
\begin{align*}
    \chi(K) &= \sum_{k=0}^{10} (b_{2k}(K) - b_{2k+1}(K)) = \sum_{k=0}^{10} ( e_1^{0,2k} - e_1^{1,2k}  + e_1^{2, 2k}) \\
    &=  \chi(\wt K) + \chi(\Omega) + \chi(\ol\Sigma)
    -\chi(\wh\Sigma) - \chi(\wt\Omega) - \chi(\ol\Omega)  + \chi(\wh\Omega) \\
    &=1208
\end{align*}
where the last equality is a straightforward computation with the dimensions in \eqref{tabellina OG6}.
\end{proof}

\begin{corollary}\label{cor Betti estimates 1 OG6}
We call \(b_i:=b_i(K)\). We have
\[
\begin{tabu}{c|c|c|c|c}
 b_0 = 1 & b_1 = 0 & b_2= 23 & b_3= 0 & 28 \le b_4 \le 198 \\ \hline 113\le b_5\le 283 &  1178 \leq  b_6\leq 1503  & b_7 \le 325 &
 171 \leq  b_8 \leq 199 & b_9 \le 28 \\ \hline
 & 7 \leq b_{10} \leq 8  & b_{11} \le 1 &  b_{12}=1 &
\end{tabu}
\]
\end{corollary}
\begin{proof}
The estimates on the even Betti numbers are the ones in Proposition \ref{prop Betti estimates 1 OG6}, computed via the dimensions in \eqref{tabellina OG6}. The only exception is the lower bound on the 4th Betti number, that is obtained thanks to the inclusion \(\Sym^2H^2(K)\hookrightarrow H^4(K)\), see \cite[Proposition 5.16]{BL18}. Observe that a similar inclusion holds for \(2k\le 6\), giving  \(\Sym^k(K)\hookrightarrow H^{2k}(K)\) hence \( b_{2k}(K) \geq \binom{6 + k}{k}\); nevertheless, this estimate turns out to be weaker (or equal) than the one in Proposition \ref{prop Betti estimates 1 OG6}. \\
For the odd Betti numbers we have used the expression in Proposition \ref{prop Betti estimates 1 OG6} combined with the estimates on the even Betti numbers.
\end{proof}
Notice that $H^5(K, \Q)$ carries a non trivial pure Hodge structure of weight $4$, but having a Hodge structure of the ``wrong" weight is not infrequent for singular varieties, as the following easy example shows.
\begin{example}\label{example}
Let $X$ be the variety obtained by gluing together 2 points $p_1, p_2$ of $\P^2$ and let $p$ be the image of $p_1$ and $p_2$ in $X$. Using the weight spectral sequence \cite[Proposition~3.3]{GNAPGP} associated to the hypercubical resolution of $X$
\[\xymatrix{
\P^2 \ar[d] & \{ p_1, p_2\}\ar[d]\ar@{_{(}->}[l]\\
X & \{ p \}\ar@{_{(}->}[l],
}
\]
one readily computes that $H^1(X, \Q) \simeq \Q$ with a pure Hodge structure of weight 0.
\end{example}

\begin{remark}
The 1st and 2nd Betti numbers were already known \cite[Theorem 1.7]{PR13} and \cite[Lemma~2.1]{BL20}.
\end{remark}

We conclude improving the estimates above with ad-hoc arguments for some of the cohomology groups of \(K\).

\begin{proposition}\label{prop Betti estimates 2 OG6}
We call \(b_i:=b_i(K)\). We have
\[
\begin{tabu}{c|c|c}
28 \le b_4\le 191 & 113\le b_5 \le 276 & 1178 \leq b_6 \leq 1502\\ \hline
 b_7 \le 324 & b_{10}=7 & b_{11}=0
\end{tabu}
\]
\end{proposition}
\begin{proof}
Regarding the 4th and 6th Betti numbers, following Proposition \ref{prop E^0,k_2} and Proposition \ref{odd-bettis OG6} we need to estimate the dimension of \[W_{2m}=\{k\in H^{2m}(\wt K)\colon  \wt j^*k\in p^*H^{2m}(\Sigma)\}\] 
for \(m=2,3\). By \cite[Theorem 1.5]{Ver96} we have an inclusion \(\Sym^mH^2(\wt K)\hookrightarrow H^{2m}(\wt K)\), hence using the decomposition \(H^2(\wt K)=\pi^*H^2(K)\oplus c_1(\sO_{\wt K}(\wt\Sigma))\cdot \Q\) we obtain
\[
\bigoplus_{i=0}^m c_1(\sO_{\wt K}(\wt\Sigma))^i\cdot \Sym^{m-i} \pi^*H^2(K)\hookrightarrow H^{2m}(\wt K)
\]
where we are assuming as convention \(\Sym^0 \pi^*H^2(K)=\Q\). The pullback \(\wt i^*_K c_1(\sO_{\wt K}(\wt\Sigma))^m\)  is non zero in \(H^{2m}(\wt\Omega)\) by Proposition \ref{restrictions}.(4), hence it is not the pullback of a class in \(H^{2m}(\Omega)\) via \(p_\Omega\) (cfr. the decomposition Proposition \ref{prop coho Omegas for OG6}.(1))  because \(\Omega\) is 0-dimensional. We conclude that \(c_1(\sO_{\wt K}(\wt\Sigma))\notin W_{2m}\), which combined with Corollary \ref{cor Betti estimates 1 OG6} gives the upper bound on the 6th Betti number and \(b_4(K)\le 198\).

For the 4th Betti number we can give a further estimate as follows. We prove that \(\wt j^* \bigl(c_1(\sO_{\wt K}(\wt\Sigma))\cdot k\bigl)\notin p^*H^{4}(\Sigma)\) for any \(k\in \pi^*H^2(K)\), hence \(c_1(\sO_{\wt K}(\wt\Sigma))\pi^*H^2(K)\cap W_4=(0)\) and \(b_4(K)\le 198-b_2(K)=191\). Note that in order to do so it is enough to prove that \(\wt j^* c_1(\sO_{\wt K}(\wt\Sigma))\notin p^*H^2(\Sigma)\) since \( j^*:H^2(K)\to H^2(\Sigma)\) is injective, cfr. Proposition \ref{prop surj of hat i, OG6}.(3). We have
\[
g^*\wt j^*\sO_{\wt K}(\wt\Sigma)=\wh j^*\gamma^* \sO_{\wt K}(\wt\Sigma)=\wh j^*\sO_{\wh K}(\wh\Sigma+2\wh\Omega_{OG})=\sO_{\wh \Sigma}(\wh\Sigma+2\wh\Omega)
\]
where the second equality was shown in Remark \ref{rem pullback Sigma tilde} and the last one holds because by definition \(\wh\Omega=\wh\Sigma\cap \wh\Omega_{OG}\). It follows that in the decomposition given in Proposition \ref{prop coho of Sigmas}.(2) we have
\[
H^2(\wh\Sigma)=\underbrace{f^*H^2(\ol\Sigma)}_{\ni c_1(\sO_{\wh\Sigma}(2\wh\Omega))}\oplus \underbrace{c_1(\sO_{\wh\Sigma}(-\wh\Sigma))\cdot f^*H^0(\ol\Sigma)}_{\ni c_1(\sO_{\wh\Sigma}(\wh\Sigma))}
\]
then the class \(c_1(g^*\wt j^*\sO_{\wt K}(\wt\Sigma))\) does not belong to \(f^* H^k(\ol\Sigma)\) hence it does not belong to \(f^*q^*H^2(\Sigma)\subseteq H^2(\wh\Sigma)\). We conclude that \(\wt j^*c_1(\sO_{\wt K}(\wt\Sigma))\notin p^*H^2(\Sigma)\) as desired.

For the 10th Betti number, looking at the dimensions in \eqref{tabellina OG6} the differential \(d_0:E_1^{0,10}\to E^{1,10}_1\) is only given by the composite morphism
\[
H^{10}(\wt K)\xrightarrow{\wt j^*} \underbrace{H^{10}(\wt \Sigma)}_{\cong \Q}\xrightarrow{g^*} \underbrace{H^{10}(\wh\Sigma)}_{\cong \Q}
\]
which is clearly surjective, hence \(\dim \ker d_0=e^{0,10}_1-e^{1,10}_1=7\) and the claim follows from Proposition \ref{odd-bettis OG6}.

The estimates on the odd Betti numbers are derived from the ones on the even Betti numbers via the relation on \(b_{2k+1}(K)\) in Proposition \ref{prop Betti estimates 1 OG6}. 
\end{proof}

Corollary \ref{cor pullback inj OG6}, Corollary \ref{cor Euler char of K}, Corollary \ref{cor Betti estimates 1 OG6} and Proposition \ref{prop Betti estimates 2 OG6} prove Theorem \ref{thm Betti of K} stated in the introduction.

\appendix 
\section{The structure of $\wh\Omega$ over $\ol\Omega$}\label{appendix}
Let \(Y=M\). Recall from Propositiomn \ref{prop global Omega's 10} that $\ol\Omega \simeq \P_\Omega(T\Omega) \to \Omega$ and that $\sL$ denotes the tautological line subbundle on $\P_\Omega(T\Omega)$. The purpose of this section is to prove the following proposition.
\begin{proposition}\label{prop appendix}
We have an isomorphism over $\ol\Omega$
$$\wh\Omega\simeq \P_{\ol\Omega}(\sL^\perp/\sL). $$
\end{proposition}
We are going to prove it in several steps.\\
Recall that for a closed subscheme $W$ of a scheme $Z$ we can consider the normal cone to $W$ in $Z$ and the projective normal one:
\begin{align*}
    &C_W Z := \underline{\mathrm{Spec}} (\bigoplus_{d\geq 0} I^d/ I^{d+1}),
    && \P(C_W Z) := \underline{\mathrm{Proj}} (\bigoplus_{d\geq 0} I^d/ I^{d+1})
\end{align*}
where $I$ is the ideal sheaf of $W$ in $Z$.
\begin{remark}
Notice that the normal cone is preserved under \'etale maps: given an \'etale map $\phi\colon\wt Z \to Z$ if we let $\wt W$ be the preimage of $W$, then for any $w\in\wt W$
$$(C_{\wt W} \wt Z)_w = (C_W Z)_{\phi (w)}.$$ 
\end{remark}

Recall that $\wh\Omega$ is the preimage of $\ol\Omega$ under the blow-up morphism $\wh M=Bl_{\ol \Sigma} \ol M \to \ol M$. By definition of blow-up, Proposition \ref{prop appendix} is then equivalent to the following.
\begin{proposition}
We have an isomorphism over $\ol\Omega$
$$\P_{\ol\Sigma} (C_{\ol\Sigma} \ol M)|_{\ol\Omega}\simeq \P_{\ol\Omega}(L^\perp/L). $$
\end{proposition}

We now recall some of O'Grady's work. Recall that the variety $M$ is the moduli space of sheaves on a K3 surface $X$ with fixed Mukai vector $(2, 0, -2)$. As any sheaf $F$ parametrised by $M$ can be obtained as a quotient
$$\sO_X(k)^{\oplus N} \to F$$
for some $N$ and $k$ which are independent of $F$, the variety $M$ is constructed as the GIT quotient by $G:=\mathrm{PGL}(N)$ of the closure $Q$ of the semistable locus $Q^{ss}$ in the Quot-scheme parametrising all the quotients of $\sO_X(k)^{\oplus N}$:
$$M = Q//G := Q^{ss}/G.$$
O'Grady in \cite[\S 1.1]{OG99} introduced a stratification of the strictly semistable locus of $Q$; we will be interested in the two closed subvarieties
\begin{align*}
\Omega_Q&:=\ol{\{ x\in Q| F_x \simeq I_Z\oplus I_Z, [Z]\in X^{[2]}\}},\\
\Sigma_Q&:=\ol{\{ x\in Q| F_x \simeq I_Z\oplus I_W, [Z],[W]\in X^{[2]}\}}.
\end{align*}
Notice that $\Omega_Q\subset \Sigma_Q$. Moreover, we have
\begin{align*}
&\Omega_Q//G = \Omega   &&\mathrm{\ and\ } && \Sigma_Q//G = \Sigma.
\end{align*}

We let $\pi_R\colon R\to Q$ be the blow-up in $\Omega_Q$, $\Sigma_R\subset R$ be the strict transform of $\Sigma_Q$ and $\pi_S\colon S\to R$ be the blow-up in $\Sigma_R$. 
Thanks to Kirwan's theory of desingularisation \cite[Theorem~(1.2.2)]{OG99}, the action of $G$ lifts to linearized actions on $R$ and $S$.
Using Luna's \'etale slice theorem, O'Grady described the normal cone to $\Sigma_R$ in $R$.
\begin{proposition}[{{\cite[(1.7.6), (1.7.12)]{OG99}}}]
Let $Z\subset X$ be a subscheme of lenght 2 and let $I_Z$ be its ideal sheaf.
Let $x\in\Omega_Q$ be a point and $F_x$ be the corresponding sheaf isomorphic to $I_Z\oplus I_Z= I_Z \otimes V$ with $V\simeq \C^2$. Let  $W:= \mathfrak{sl}(V)$, then
$$ \pi_R^{-1}(x)\cap\Sigma_R^{ss} \simeq \P\{\varphi\in \Hom(W,\Ext^1(I_Z, I_Z)): \rk(\varphi) \leq 1,\ \varphi\  \mathrm{semistable}\}.$$
Moreover, let $[\varphi]\in \pi_R^{-1}(x)\cap\Sigma_R^{ss}$ for some $\varphi\in \Hom(W,\Ext^1(I_Z, I_Z))$, let $St([\varphi])$ be the stabilizer and $\omega_\varphi$ be the symplectic form induced on $\Im \varphi^\perp/\Im \varphi$ by the symplectic form on $\Ext^1(I_Z, I_Z)$. The normal cone of $\Sigma_R$ to $R$ at the point $[\varphi]$ is isomorphic to the normal cone of $\Sigma_R\cap\Omega_R$ to $\Omega_R$ and 
there exists a $St([\varphi])$-equivariant isomorphism 
$$(C_{\Sigma_R\cap\Omega_R} \Omega_R)_{[\varphi]} \xrightarrow{\sim}
\{ y\in \Hom(\ker \varphi, \Im \varphi^\perp/\Im \varphi): y^* \omega_\varphi = 0  \}.
$$
Further, $St([\varphi]) = \O(\ker\varphi)$.
\end{proposition}

Since we want to use O'Grady's result, we first relate the variety $M, \ol M$ and $\wh M$, with $Q,R$ and $S$. The blow-up morphisms $S \to R\to Q$ descend to the quotients:
\[S//G \to R//G \to Q//G. \]
\begin{proposition}\label{prop quot of the varieties M}
We have canonical isomorphisms
\[\xymatrix{
\wh M \ar[r]^{\sim}\ar[d] & S//G\ar[d] \\
\ol M \ar[r]^{\sim}\ar[d] & R//G\ar[d] \\
M \ar@{=}[r] & Q//G.
}
\]
Moreover, under these identifications we have
\begin{align*}
&\ol\Omega_{OG} =\Omega_R //G,
    &&\ol\Sigma = \Sigma_R//G &&&\mathrm{and} 
    &&&&    \ol\Omega = \Sigma_R\cap\Omega_R//G
\end{align*}
where $\Omega_R$ denotes the exceptional divisor of $R\to Q$.
\end{proposition}
\begin{proof}
Consider the composition $R^{ss} \to Q^{ss} \to M$, under which the preimage of $\Omega$ is equal to $\Omega_R$ which is a Cartier divisor. Thus, by the universal property of the blow-up we get a morphism $R^{ss}\to \ol{M}$. As it is $G$-invariant, it factors through the quotient $R//G$, which is an isomorphism on the complement of the exceptional divisor $\Omega_{OG}$ and maps $\Omega_R//G$ onto $\ol\Omega_{OG}$.

We want to show that the surjection $\Omega_R//G \to \ol\Omega_{OG}$ is an isomorphism, in other words we are left to prove the following
\begin{claim}
Let $I$ be the ideal of $\Omega_Q$ in $Q$ and $J$ be the ideal of $\Omega$ in $M$. The morphism
  $$\underline{\mathrm{Proj}} (\bigoplus_{d\geq 0} I^d/ I^{d+1})//G = \Omega_R//G \to \ol \Omega_{OG}
 = \underline{\mathrm{Proj}} (\bigoplus_{d\geq 0} J^d/ J^{d+1}).$$
 is an isomorphism.
 \end{claim}

Using Luna's \'etale slice theorem \cite[Theorem~(1.2.1)]{OG99}  we first reduce to deal with an affine quotient. Indeed, for any $\omega\in\Omega_R$ we can find a $St(\omega)$-stable affine open set $\omega\in\sV\subset Q$, such that the multiplication morphism 
\[
 G\times_{St(\omega)} \sV \xrightarrow{\phi} Q
\]
has open image, is \'etale over its image and is $G$-equivariant (here $St(\omega)$ acts on $G\times_{St(\omega)} \sV$ by $h(g,p):=(gh^{-1}, hp)$ for any $p\in\sV$).
Let $\pi\colon \sV \to \ol\sV := \sV // St(\omega)$ be the quotient and $\ol\omega$ be the image of $\omega$. The quotient map
\[\ol\phi\colon \ol\sV  \to Q//G = \ol M\]
has open image and is \'etale over its image. For such a $\sV$ we have that \cite[(1.2.2)]{OG99}
\[(C_{\sV\cap\Omega_Q} \sV)_{\omega} \simeq
(C_{\Omega_Q} Q)_{\omega}.
\]
As $\ol\phi$ is \'etale we are left to prove 
\[
\P (C_{\sV\cap\Omega_Q} \sV)_{\omega}//St(\omega) \simeq
\P (C_{\ol\sV\cap\pi(\Omega_Q)} \ol\sV)_{\ol\omega}.
\]
Let $I$ be the ideal of $\sV\cap\Omega_Q$ in $\sV$. As $St(\omega)$ is reductive \cite[Corollary~(1.1.8)]{OG99}, then $I^{St(\omega)}$ equals the ideal of $\sV//St(\omega) \cap \pi(\Omega_Q)$ in $\sV// St(\omega)$. Thus we have proven at once that $\Omega_R//G \simeq \ol\Omega_{OG}$ and that $\ol M \simeq R//G$.\bigskip

By construction we have an isomorphism $(\Sigma_R\smallsetminus \Omega_R)//G \xrightarrow{\sim} \ol\Sigma \smallsetminus\ol\Omega$ and we conclude $\Sigma_R//G = \ol \Sigma$ by irreducibility.
Thanks to this last equality one shows with an analogous reasoning $S//G\simeq \wh M$.

Finally,
\[ (\Sigma_R \cap \Omega_R)//G = \Sigma_R //G \cap \Omega_R//G = \ol\Sigma \cap \ol\Omega_{OG} = \ol\Omega, \]
where the first equality holds because $R^{ss}/G$ is a good quotient.

\end{proof}

\begin{lemma}\label{lemmino}
Let $x\in\Omega_Q$, let $[\varphi]$ be a point in $\pi_R^{-1}(x)\cap \Sigma_R^{ss}$ and $\ol\varphi$ its image under the quotient map $\Omega_R\cap \Sigma_R^{ss} \to \ol\Omega$. Then 
\[
\P(C_{\ol\Sigma} \ol M)_{\ol\varphi} \simeq
\P(C_{\Sigma_R} R)_{[\varphi]}//St([\varphi]).
\]
\end{lemma}
\begin{proof}
Using the results in Proposition \ref{prop quot of the varieties M} and that $St([\varphi]) = \O(\ker\varphi)$ is also reductive, the proof is completely similar to the proof above.
\end{proof}

\begin{proof}[Proof of Proposition \ref{prop appendix}]
Notice that the above mentioned description of the cone due to O'Grady works in an analytic neighbourhood of any point $[\varphi] \in \Omega_R \cap \Sigma^{ss}_R$ and that in such a neighbourhood we have the natural isomorphism
$$\P (C_\Sigma R)_{[\varphi]} // St([\varphi]) \to \P(\sL^\perp/\sL)_{\ol\varphi},\  [y] \mapsto [\Im y].$$
Combining this with Lemma~\ref{lemmino} we get the desired isomorphism 
$$
\P(C_{\ol\Sigma} \ol M)|_{\ol\Omega} \to \P(\sL^\perp/\sL).
$$

\end{proof}

\bibliography{literatur}
\bibliographystyle{alpha}
\end{document}